\documentclass[10pt,a4paper]{scrartcl}%

\pdfoutput=1

\usepackage[T1]{fontenc}
\usepackage{epsfig}
\usepackage{graphicx}
\usepackage{amssymb}
\usepackage{amsmath}
\usepackage[utf8]{inputenc}
\usepackage{amsfonts}
\usepackage{amsthm}
\usepackage{stmaryrd}
\usepackage{enumitem}

\usepackage{tikz}
\usepackage{tikz-qtree}
\usetikzlibrary{fit,calc,positioning,decorations.pathreplacing,matrix,trees,patterns}
\usepackage{xcolor}

\usepackage{geometry}
\usepackage{hyperref}
\usepackage[numbers]{natbib}

\makeatletter

\@addtoreset{equation}{section}
\makeatother

\theoremstyle{plain}
\newtheorem{theorem}{Theorem}[section]

\newtheorem{lemma}[theorem]{Lemma}

\newtheorem{proposition}[theorem]{Proposition}

\theoremstyle{definition}
\newtheorem{definition}[theorem]{Definition}

\theoremstyle{remark}

\DeclareMathOperator*{\argmin}{argmin}
\DeclareMathOperator*{\arccosh}{arcosh}

\newcommand{\RST}{\text{RST}}
\newcommand{\DSF}{\text{DSF}}

\newcommand{\CFD}{\text{CFD}}
\newcommand{\MBD}{\text{MBD}}
\newcommand{\CO}{\text{CO}}
\newcommand{\Cyl}{\text{Cyl}}

\newcommand{\Stab}{\text{Stab}}
\newcommand{\Cone}{\text{Cone}}
\newcommand{\Cl}{\text{Cl}}
\newcommand{\Rad}{\text{Rad}}
\newcommand{\Vois}{\text{Vois}}
\newcommand{\Stop}{\text{Stop}}
\newcommand{\hit}{\text{hit}}
\newcommand{\Reg}{\mbox{Reg}}
\newcommand{\Vol}{\mbox{Vol}}
\newcommand{\Activ}{\mbox{Inhib}}
\newcommand{\fluct}{\text{fl}}
\newcommand{\decay}{\text{dec}}
\newcommand{\ball}{\text{ball}}
\newcommand{\geom}{\text{geom}}

\newcommand{\dist}{\text{dis}}
\newcommand{\ind}{\mathbf{1}}
\makeindex

\newcommand{\tvc}[1]{#1}
\newcommand{\dave}[1]{#1}

\title{Hyperbolic Radial Spanning Tree\footnote{Acknowledgments. This work has been supported by the LAMAV (Université Polytechnique des Hauts de France) and the Laboratoire P. Painlevé (Université de Lille). It has also benefitted from the GdR GeoSto 3477 from CNRS, the Labex CEMPI (ANR-11-LABX-0007-01), the Labex B\'ezout (ANR-10-LABX-0058) and the ANR PPPP (ANR-16-CE40-0016).}}

\author{David Coupier \footnote{Institut Mines T\'el\'ecom Nord Europe, F-59650 Villeneuve d'Ascq, France \texttt{david.coupier@imt-nord-europe.fr}}, \quad Lucas Flammant \footnote{Univ. de Valenciennes, CNRS, EA 4015 - LAMAV, F-59313 Valenciennes Cedex 9, France} \quad and \quad Viet Chi Tran \footnote{LAMA, Univ Gustave Eiffel, Univ Paris Est Creteil, CNRS, F-77454 Marne-la-Vall\'ee, France; IRL 3457, CRM-CNRS, Université de Montréal, Canada. 
\texttt{chi.tran@univ-eiffel.fr}}}
\date{\today}

\begin{document}

\maketitle

\begin{abstract}
We define and analyze an extension to the $d$-dimensional hyperbolic space of the Radial Spanning Tree (RST) introduced by Baccelli and Bordenave in the two-dimensional Euclidean space (2007). In particular, we will focus on the description of the infinite branches of the tree. The properties of the two-dimensional Euclidean RST are extended to the hyperbolic case in every dimension: almost surely, every infinite branch admits an asymptotic direction and each asymptotic direction is reached by at least one infinite branch. Moreover, the branch converging to any deterministic asymptotic direction is unique almost surely. To obtain results for any dimension, a completely new approach is considered here. \tvc{Our strategy mainly involves the two following ingredients, that rely on the hyperbolic Directed Spanning Forest (DSF) introduced and studied in Flammant (2019).} First, the hyperbolic metric allows us to obtain fine control of the branches' fluctuations in the hyperbolic DSF without using planarity arguments. Then, we couple the hyperbolic RST with the hyperbolic DSF and conclude.
\end{abstract}

\textbf{Key words: } continuum percolation, hyperbolic space, stochastic geometry, random geometric tree, Radial Spanning Tree, Directed Spanning Forest, Poisson point processes.
\bigbreak
\textbf{AMS 2010 Subject Classification:} Primary 60D05, 60K35, 82B21.\\




\section{Introduction}

Geometric random trees are well studied in the literature since they interact with many other fields, such as communication networks, particles systems or population dynamics. Several works have established scaling limits for two-dimensional radial trees \cite{coupier2019directed, coletti2013radial} and translation invariant forests \cite{dsftobw,roy2016random,gangopadhyay}. In addition, random spanning trees appear in the context of first passage percolation \cite{howard}. A complete introduction to geometric random graphs is given in Penrose \cite{penrose2003random}.

\dave{Here we are interested in the} Radial Spanning Tree (RST), \dave{introduced by Baccelli and Bordenave \cite{baccelli} in the Euclidean plane and with motivations from communication networks.} The construction of this tree is the same on the plane $\mathbb{R}^2$ or on the hyperbolic space $\mathbb{H}^{d+1}$ (presented below). The set of vertices is given by a homogeneous Poisson Point Process (PPP) $\mathcal{N}$ of intensity $\lambda$. The RST rooted at the origin $0$ is the graph obtained by connecting each point $z \in \mathcal{N}$ to its parent $A(z)$, defined as the closest point to $z$ among all points $z' \in \mathcal{N} \cup \{0\}$ that are closer to the origin than $z$. This defines a random tree rooted at the origin with a radial structure. Given a path, we will say that the forward direction is towards $0$ and the backward direction is towards infinity. An infinite backward path is defined as a sequence of Poisson points $(z_n)_{n \ge 0} \in \left(\mathcal{N} \cup \{0\} \right)^{\mathbb{N}}$ with $z_0=0$ and $z_n=A(z_{n+1})$ for any $n \ge 0$. \dave{Because a vertex can be the ancestor of no other vertex, all backward paths are not infinite.} 

The topological properties of the bi-dimensional Euclidean RST are well-understood. Baccelli and Bordenave showed that almost surely (a.s.), any infinite backward path admits an asymptotic direction, \tvc{i.e. that a.s. for any infinite path $(z_n)_{n\geq 0}$, the sequence $(z_n/|z_n|)_{n\geq 0}$ converges to a limit in the unit sphere $\mathbb{S}^1$ of $\mathbb{R}^2$. The limit is called the asymptotic direction of the infinite path. Moreover, a.s., every asymptotic direction of $\mathbb{S}^1$ is reached by at least one infinite backward path and there exists a.s. a unique infinite path in any given deterministic asymptotic direction \cite{baccelli}.} These results on the infinite paths are completed by Baccelli, Coupier \& Tran \cite{baccelli2013semi}.

\medbreak

For any integer $d \ge 1$, the hyperbolic space $\mathbb{H}^{d+1}$ is a $(d+1)$-dimensional Riemannian manifold with constant negative curvature, that can be chosen equal to $-1$ without loss of generality. It admits a set of ideal boundary points $\partial \mathbb{H}^{d+1}$, and $\overline{\mathbb{H}^{d+1}}:=\mathbb{H}^{d+1} \cup \partial \mathbb{H}^{d+1}$ denotes the hyperbolic space endowed with its boundary. It is a non-amenable space, i.e. the measure of the boundary of a large subset is not negligible with respect to its volume. \tvc{The hyperbolic space is defined in more details in \cite{cannon,chavel,paupert,ratcliffe}.}

There is a growing interest for the study of random models in a hyperbolic setting. Benjamini and Schramm establish percolation results on regular tilings and Voronoï tessellations in the hyperbolic plane \cite{benjamini}. Mean characteristics of the Poisson-Voronoï tessellation have also been considered in a general Riemannian manifold by Calka et al. \cite{calka2018mean}. This interest is explained by at least two reasons. First, hyperbolic random graphs are well-fitted to model social networks \cite{socialnetworks}. Secondly, strong differences have been noticed for properties of random models depending on whether they are considered in an Euclidean or in a hyperbolic setting. For example, some hyperbolic random graphs admit a non-degenerate regime with infinitely many unbounded components in the hyperbolic space \cite{tykesson, hutchcroft2019percolation}, which is generally not the case in the Euclidean space. In addition, behaviors of non-amenable spaces are well studied in a discrete context \cite{percolationbeyondZd, lyons2017probability, pete2014probability}.

Thus it is natural to consider and study the hyperbolic RST, which we define in the same way as the Euclidean RST. A simulation of the two-dimensional hyperbolic RST is given in Figure \ref{Fig:simulationRST}. In this paper, we extend the results of Baccelli and his coauthors to hyperbolic geometry \textbf{in every dimension}. Here is our main result:

\begin{theorem}
\label{Thm:mainthm}
For any dimension $d \ge 1$ and any intensity $\lambda$, the following happens:
\begin{enumerate}[label=(\roman*)]
    \item almost surely, any infinite backward path $(z_n)_{n \in \mathbb{N}}$ admits an asymptotic direction, i.e. there exists $z_\infty \in \partial \mathbb{H}^{d+1}$ such that $\lim_{n \to \infty} z_n=z_\infty$ (in the sense of the topology of $\overline{\mathbb{H}^{d+1}}$);
    \item almost surely, for any $\dave{z_\infty} \in \partial \mathbb{H}^{d+1}$, there exists an infinite backward path $(z_n)$ with asymptotic direction $\dave{z_\infty}$; 
    \item for any deterministic boundary point $\dave{z_\infty} \in \partial \mathbb{H}^{d+1}$, the path with asymptotic direction $\dave{z_\infty}$ is almost surely unique;
    \item the set of boundary points with at least two infinite backward paths is dense in $\partial \mathbb{H}^{d+1}$;
    \item this set is moreover countable in the bi-dimensional case (i.e. $d=1$).
\end{enumerate}
\end{theorem}

\tvc{Theorem \ref{Thm:mainthm} describes the infinite branches of the hyperbolic RST: every infinite branch admits an asymptotic direction and Point (ii) and (iii) say that for any fixed and deterministic boundary point $I$, there exists a unique infinite path having $I$ as asymptotic direction. But there is a random dense set of boundary points having more than one backwards paths, and in dimension 2 (for $d=1$) this set is countable.}\\

Establishing the results announced in Theorem \ref{Thm:mainthm} in every dimension constitutes the main originality of this paper. For the two reasons explained further, the proofs of Baccelli and Bordenave in the 2D-Euclidean setting \cite{baccelli} cannot be generalized to higher dimensions.

In both contexts $\mathbb{R}^2$ and $\mathbb{H}^{d+1}$, for any $d \ge 1$, the proofs of (i), (ii), (iv) and (v) of Theorem \ref{Thm:mainthm} follow the strategy of Howard and Newman \cite{howard}, which is to show that the tree is \emph{straight}, that is, the \tvc{descendant} subtree of a vertex far from the origin is included in a thin cone. To prove that the 2D-Euclidean RST is straight, Baccelli and Bordenave used a translation invariant model derived from the RST: \tvc{the Directed Spanning Forest (DSF), which constitutes a local approximation of the RST far from the origin \cite{baccelli}.} They exploit the theory of Markov chains to bound from above fluctuations of trajectories in the DSF and then, they deduce the straightness of the RST via planarity. This strategy cannot be generalized to higher dimensions. However, in $\mathbb{H}^{d+1}$, we manage to control the angular deviations of branches in the $\RST$ without resorting to an auxiliary model, that requires planarity as in the Euclidean setting. The hyperbolic metric guarantees that angular deviations decay exponentially fast with the distance to the origin, which is strong enough to show straightness.


In addition, in the Euclidean context, the uniqueness part (point (iii) in Theorem \ref{Thm:mainthm}) is only proved in dimension $2$ since it strongly uses planarity \cite{howard, baccelli}, and the strategy of proof cannot be generalized to higher dimensions. To prove (iii) in $\mathbb{H}^{d+1}$, our strategy consists in exploiting the link existing between the hyperbolic RST and the hyperbolic DSF, defined and studied in Flammant 2019 \cite{dsf}, which is the hyperbolic counterpart of the Euclidean DSF used by Baccelli and Bordenave. Roughly speaking, the hyperbolic DSF can be defined as the limit of the hyperbolic RST when the origin point tends to an ideal boundary point. Similarly to the Euclidean setting, it constitutes a local approximation of the RST far from the origin. The proof of (iii) exploits the coalescence of the hyperbolic DSF (i.e. it is almost surely a tree) \cite[Theorem 1.1]{dsf}, which is a non-trivial fact obtained by exploiting the mass-transport principle, and a local coupling between the two models.\\

After defining the hyperbolic RST and giving its basic properties, we define two quantities that encode angular fluctuations along trajectories, the Cumulative angular Forward Deviations ($\CFD$) and the Maximal Backward Deviations (MBD). We then establish upper bounds of these quantities: first, we upper-bound the Maximal Backward Deviations in a thin annulus of width $\delta>0$ (Proposition \ref{Prop:controlannulus}) and then we deduce a global control of $\MBD$ in the whole space (Proposition \ref{Prop:globalfluct}), that roughly says that angular deviations decay exponentially fast with the distance to the origin. From this upper-bound, we deduce that the RST is straight in the sense of Howard \& Newman (Proposition \ref{Prop:straightness}). The points (i), (ii), (iv) and (v) in Theorem \ref{Thm:mainthm} can be deduced from straightness and the upper-bound of $\MBD$ given by Proposition \ref{Prop:globalfluct}. The point (iii) (the uniqueness part) is done by exploiting a local coupling existing between the RST and the DSF far from the origin.
\medbreak
The rest of paper is organized as follows. In Section \ref{S:firstdef}, we set some reminders of hyperbolic geometry and we define the hyperbolic RST. Then, we give its basic properties and a road-map of the proofs. We also announce the upper bounds of angular deviations (Propositions \ref{Prop:controlannulus} and \ref{Prop:globalfluct}) and the straightness property (Proposition \ref{Prop:straightness}). The proof of Theorem \ref{Thm:mainthm} is done in Section \ref{S:proofmainthm}. Proposition \ref{Prop:controlannulus} is proved in Section \ref{S:proofcontrolannulus} and the proofs of Propositions \ref{Prop:globalfluct} and \ref{Prop:straightness} are done in Section \ref{S:proofglobalfluct}.

\begin{figure}[!h]
    \centering
    \includegraphics[scale=0.7]{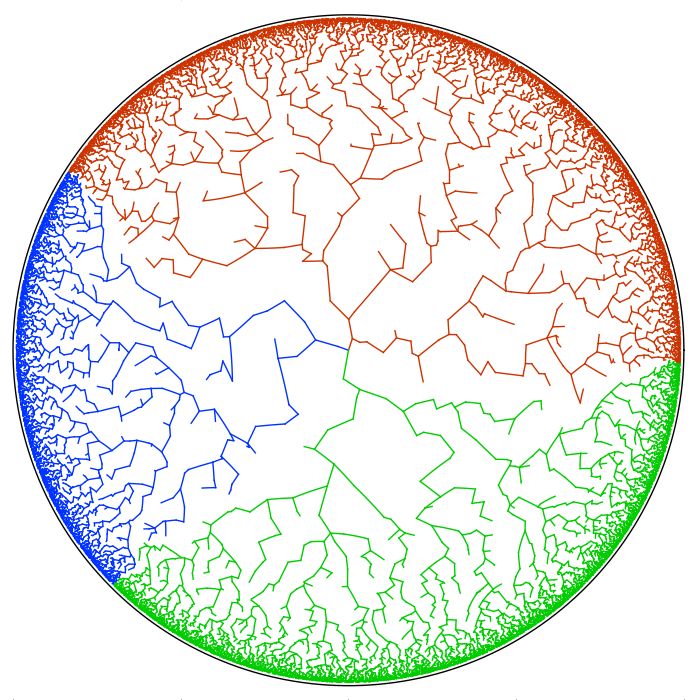}
    \caption{Simulation of the two-dimensional hyperbolic RST, with $\lambda=30$, in the Poincaré disc model. The edges are represented by geodesics. The different connected components of the RST (apart from the root) are represented with different colors.}
    \label{Fig:simulationRST}
\end{figure}

\section{Definitions, notations and basic properties}
\label{S:firstdef}

We denote by $\mathbb{N}$ the set of non-negative integers and by $\mathbb{N}^*$ the set of positive integers. In the rest of the paper, $c$ (resp. $C$) will be some small (resp. large) constant whose value can change from a line to another.

\subsection{The hyperbolic space}

We refer to \cite{ratcliffe} for a complete introduction to hyperbolic geometry. For $d \in \mathbb{N}^*$, the $(d+1)$-dimensional hyperbolic space, denoted by $\mathbb{H}^{d+1}$, is a $(d+1)$-dimensional Riemannian manifold of constant negative curvature $-1$ that can be defined by several isometric models. One of them is the open-ball model consisting in the unit open ball
\begin{eqnarray}\label{def:Poincare-disk}
I=\{(x_1,...,x_{d+1}) \in \mathbb{R}^{d+1},~x_1^2+...+x_{d+1}^2<1\}
\end{eqnarray}
endowed with the following metric:
\begin{eqnarray}\label{def:Poincare-disk-dist}
\dave{ds_I^2:=4\frac{dx_1^2+...+dx_{d+1}^2}{(1-x_1^2-...-x_{d+1}^2)^2}.}
\end{eqnarray}
\tvc{We denote by $d(\cdot,\cdot)$ the hyperbolic distance in $\mathbb{H}^{d+1}$, and by $\|\cdot\|$ the Euclidean norm in $\mathbb{R}^d$, with the convention $\|\infty\|=\infty$. }

The volume measure on $(I,ds_I^2)$, denoted by $\Vol_I$, is given by
\begin{eqnarray}
\dave{d\Vol_I=2^{d+1}\frac{dx_1...dx_{d+1}}{(1-x_1^2-...-x_{d+1}^2)^{d+1}}.}
\end{eqnarray}

\tvc{An important fact about hyperbolic geometry is that $\mathbb{H}^{d+1}$ is homogeneous, isotropic and rotation invariant. It means that the group of isometries of $\mathbb{H}^{d+1}$ acts transitively on the unit tangent bundle of $\mathbb{H}^{d+1}$: given two points $x,y \in \mathbb{H}^{d+1}$ and two unit tangent vectors $u \in T_x\mathbb{H}^{d+1},~v \in T_y\mathbb{H}^{d+1}$, there exists an isometry $g$ of $\mathbb{H}^{d+1}$ such that $g(x)=y$ and that pushes forward $u$ on $v$. The notations $T_x$, $T_y$ and the vocabulary relating to Riemannian geometry are defined in \cite{leeriemannian}. We refer to \cite[Proposition 1.2.1 p.5]{paupert} for a proof. }\\

\tvc{Let $0 \in \mathbb{H}^{d+1}$ be some arbitrary origin point (it can be thought as the center of the ball in the open-ball representation), which will play the role of the root of the RST.} \\
\tvc{The hyperbolic space $\mathbb{H}^{d+1}$ is naturally equipped with a set of points at infinity, denoted by $\partial \mathbb{H}^{d+1}$. In the open-ball model $(I,ds_I^2)$, the set of points at infinity is identified by the boundary unit sphere. Let us denote by $\mathbb{S}^d$ the unit Euclidean sphere in $\mathbb{R}^{d+1}$ and by $\nu$ its $d$-dimensional volume measure. We denote by $\overline{\mathbb{H}^{d+1}}:=\mathbb{H}^{d+1} \cup \partial \mathbb{H}^{d+1}$ the hyperbolic space $\mathbb{H}^{d+1}$ plus the set of points at infinity, with the topology given by the closed ball. A point $z_\infty \in \partial \mathbb{H}^{d+1}$ is called \emph{ideal point} or \emph{point at infinity}.} \\
The metric becomes smaller as we get closer to the boundary unit sphere $\partial I$, and this boundary is at infinite distance from the center $0$.\\
For any subset $E \subset \overline{\mathbb{H}^{d+1}}$, $\overline{E}$ denotes the closure of $E$ in $\overline{\mathbb{H}^{d+1}}$. \\

In the open ball model $(I,ds_I^2)$, the geodesics are of two types: the diameters of $I$ and the arcs that are perpendicular to the boundary unit sphere $\mathbb{S}^d$, \tvc{see Figure \ref{fig:openball}.} We refer to discussion \cite[p.80]{cannon} for a proof. This model is conformal, which means that the hyperbolic angle between two geodesics corresponds to their Euclidean angle in the open-ball representation. \tvc{For $z_1,z_2,z_3 \in \overline{\mathbb{H}}^{d+1}$, $\widehat{z_1 z_2 z_3}$ is the measure of the corresponding (non-oriented) angle. }\\

\begin{figure}[!h]
    \centering
    \includegraphics[scale=0.7]{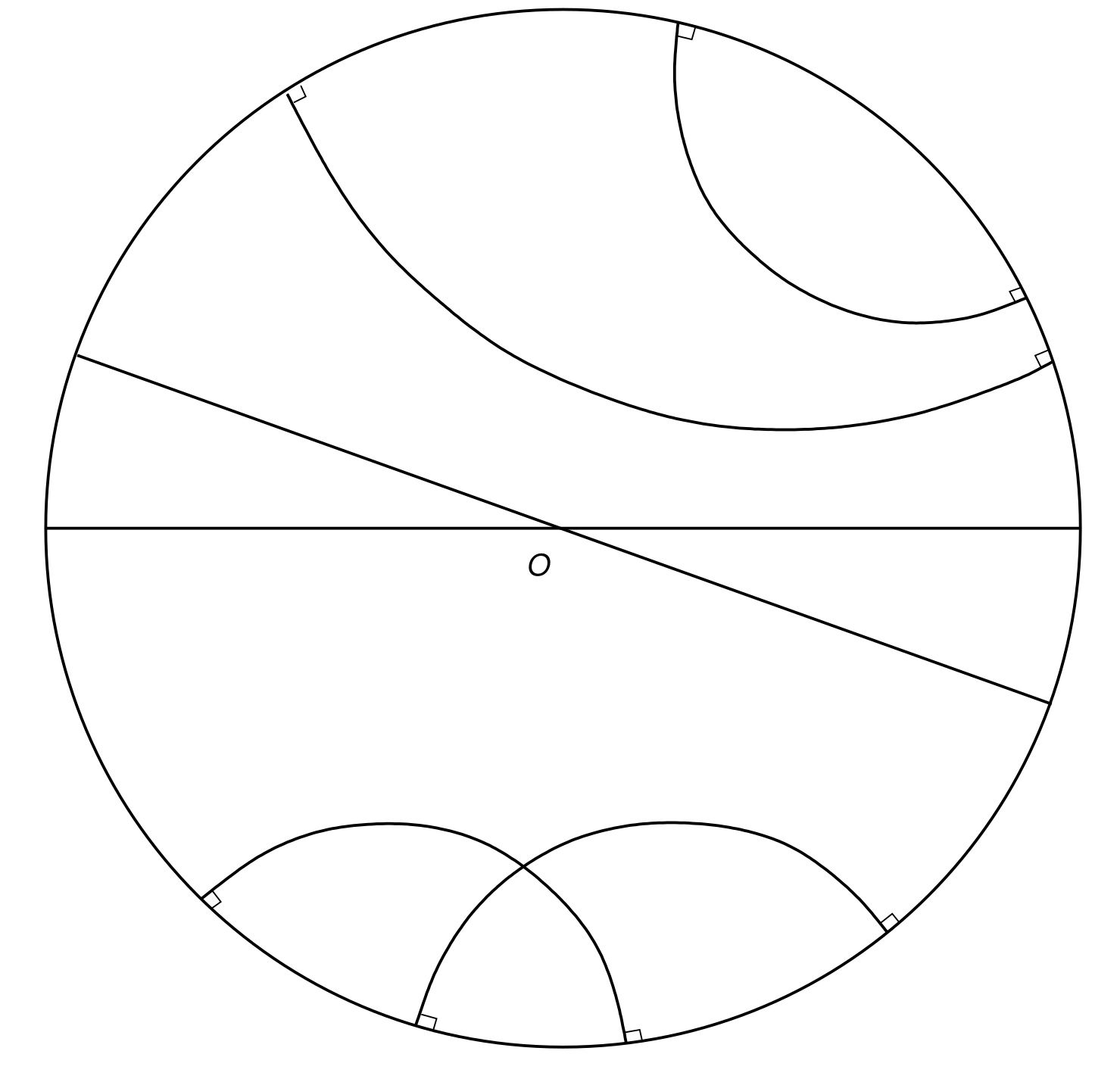}
    \caption{Geodesics in the open ball model}
    \label{fig:openball}
\end{figure}

For $z_1,z_2 \in \overline{\mathbb{H}}^{d+1}$, let us denote by $[z_1,z_2]$ the geodesic between $z_1$ and $z_2$. Moreover, we set the notations:
\[[z_1,z_2[:=[z_1,z_2] \backslash\{z_2\}, \quad ]z_1,z_2]:= [z_1,z_2] \backslash\{z_1\} ,\quad ]z_1,z_2[ \backslash(\{z_2\} \cup \{z_2\}).\]
Let us denote by $[z_1,z_2)$ (resp. $(z_1,z_2]$) the semi-geodesic passing \tvc{through} $z_2$ (resp. $z_1$) and ending at $z_1$ (resp. $z_2$). \\

For $z \in \mathbb{H}^{d+1}$ and $r>0$, we denote by $B(z,r):=\{z' \in H,~d(z,z')<r\}$ (resp. $S(z,r):=\{z' \in H,~d(z,z')=r\}$) the hyperbolic ball (resp. $sphere$) centered at $z$ of radius $r$, and we set $B(r):=B(0,r)$ \tvc{(resp. $S(r):=S(0,r)$).} For $x \in \mathbb{R}^d$ and $r>0$, let us also denote by $B_{\mathbb{R}^d}(x,r):=\{x' \in \mathbb{R}^d,~\|x'-x\|<r\}$ the Euclidean ball centered at $x$ of radius $r$.\\
\tvc{For any point $z \in \overline{\mathbb{H}^{d+1}}$ and $\theta>0$, $\Cone(z,\theta):=\{z' \in \mathbb{H}^{d+1},~\widehat{z 0 z'} \le \theta\}$ is defined as the cone of apex $0$ and aperture $\theta$ (if $\theta \ge \pi$ then $\Cone(z,\theta)$ is the whole space $\mathbb{H}^{d+1}$).}\\

Since the RST is a graph rooted at $0$, a convenient way to represent points in $\mathbb{H}^{d+1}$ is to use polar coordinates. Recall that $0$ is the origin point. For any point $z \in \mathbb{H}^{d+1}$, we denote by $z=(r;u)$ its polar coordinates w.r.t. $0$: $r$ is its distance to $0$ and $u \in UT_{0}\mathbb{H}^{d+1} \simeq \mathbb{S}^d$ is its direction ($UT_{0}\mathbb{H}^{d+1}$ is the unitary tangent space of $0$ in $\mathbb{H}^{d+1}$, \tvc{consisting of tangent vectors of norm 1}). In polar coordinates, the volume measure $\Vol$ is given by
\begin{eqnarray}
\label{E:dvolpolar}
d\Vol(r;u)=\sinh(r)^d~dr~d\nu(u).
\end{eqnarray}
A direct consequence is that the volume of a ball of radius $r$ is given by:
\begin{eqnarray}
\label{E:volball}
\Vol(B(r))=\int_0^r \sinh(t)^d~dt \times \nu(\mathbb{S}^d) \asymp e^{dr} \text{ when } r \to \infty,
\end{eqnarray}
\tvc{where $\asymp$ means that $\Vol(B(r))$ is asymptotically proportional to $e^{dr}$ when $r\rightarrow +\infty$. We refer the reader to \cite[Section III.3 and (III.4.1)]{chavel}.} \\

The hyperbolic law of cosines \dave{\cite[(6.3-5) p.197]{ramsayrichtmyer}} is a well adapted tool to compute distances using polar coordinates. Given $z_1=(r_1;u_1)$, $z_2=(r_2,u_2) \in \mathbb{H}^{d+1}$, the hyperbolic law of cosines gives,
\begin{eqnarray}
\label{E:cosinelaw}
\cosh d(z_1,z_2)=\cosh(r_1)\cosh(r_2)-\langle u_1,u_2 \rangle \sinh(r_1)\sinh(r_2).
\end{eqnarray}

\subsection{The hyperbolic RST}




In the rest of the paper, the dimension $d\geq 1$ and the intensity $\lambda>0$ are fixed. Let $\mathcal{N}$ be a homogeneous PPP of intensity $\lambda$ in $\mathbb{H}^{d+1}$. The definition of the hyperbolic RST is similar to the Euclidean case. The set of vertices is $\mathcal{N} \cup \{0\}$. Each vertex $z \in \mathcal{N}$ \dave{has a unique outgoing edge that connects $z$} to the closest Poisson point among those that are closer to the origin:

\begin{definition}[Radial Spanning Tree in $\mathbb{H}^{d+1}$]
\label{Def:rsthyp}
For any $z=(r;u) \in \mathcal{N}$, the \emph{parent (or ancestor)} of $z$ is defined as
$$
A(z) := \argmin_{z' \in (\mathcal{N}\cup\{0\}) \cap B(r)} d(z',z) ~.
$$
We call \emph{Radial Spanning Tree (RST)} in $\mathbb{H}^{d+1}$ rooted at $0$ the oriented graph $(V,\vec{E})$ where
$$V:=\mathcal{N} \cup \{0\}, \quad \vec{E}:=\{(z,A(z)),~z \in \mathcal{N}\}.$$
\end{definition}

It is possible to assume that $\mathcal{N} \cup \{0\}$ does not contain isosceles triangles, since this event has probability $1$. Thus the ancestor $A(z)$ is well-defined, \tvc{and clearly the RST is a tree rooted at $0$.}

For $z \in \mathcal{N} \cup \{0\}$ and $k \in \mathbb{N}$, \tvc{let us define the $k$-th ancestor of $z$ by $A^{(k)}(z)=A \circ..\circ A(z)$, where composition is taken $k$ times, and their descents after $k$ generations by
$A^{(-k)}(z)=\{z' \in \mathcal{N},~A^{(k)}(z')=z\}$ (in particular $A^{(-1)}(z)$ is the set of daughters of $z$).} For $z \in \mathcal{N}$ and $r \ge 0$, let us define 
\begin{equation}\label{def:B+}
B^+(z,r):=B(z,r) \cap B(0,d(0,z)) \text{ and }B^+(z):=B^+(z,d(z,A(z))).
\end{equation} 
By definition of the parent, $B^+(z) \cap \mathcal{N}=\emptyset$ for all $z \in \mathcal{N}$. \dave{This fact is responsible for many difficulties when studying the RST. Indeed, when restarting from $A(z)=(r';u')$ and constructing the path forward (towards $0$), with probability one $B(r')\cap B^+(z,d(z,A(z)))\not=\emptyset$. This means that the geometric information used to determine $A(z)$ is still involved for later steps of the process, generating statistical dependencies. }\\
\dave{Nevertheless, we can obtain the following basic properties about RST proved in Appendix \ref{S:firstprop}.

\begin{proposition}
\label{Prop:firstprop}
The RST is a tree and it has finite degree a.s. Moreover, in the bi-dimensional case ($d=1$), the representation of the RST obtained by connecting each vertex $z \in \mathcal{N}$ to its parent $A(z)$ by the geodesic $[z,A(z)]$ (instead of $[z,A(z)]^*$) is planar, i.e. their is no two points $z_1,z_2 \in \mathcal{N}$ such that $[z_1,A(z_1)] \cap [z_2,A(z_2)] \neq \emptyset$.
\end{proposition}
}

\dave{In the above proposition, the edges are geodesics. However, }Definition \ref{Def:rsthyp} does not specify the shape of edges. Since the results announced in Theorem \ref{Thm:mainthm} only concern the graph structure of the hyperbolic RST, their veracity does not depend on the geometry of edges. Although it is more natural to represent edges with hyperbolic geodesics, we do another choice in the sequel which will appear more convenient for the proofs. Given $z_1=(r_1;u_1),z_2=(r_2;u_2) \in \mathbb{H}^{d+1}$ such that $0 \notin [z_1,z_2]$, we define a path $[z_1,z_2]^*$, in an isotropic way, verifying the two following conditions:
\begin{enumerate}[label=\roman*)]
    \item the distance to the origin $0$ is \tvc{monotone} along the path $[z_1,z_2]^*$,
    \item the distance to $z_1$ is also \tvc{monotone} along this path.
\end{enumerate}
It will be necessary for the proofs that the shape of edges satisfy conditions (i) and (ii): remark that the geodesic $[z_1,z_2]$ or the Euclidean segment between $z_1$ and $z_2$ do not satisfy condition (i) in general. Since $0 \notin [z_1,z_2]$, $u_1$ and $u_2$ are not antipodal. Thus one can consider the unique geodesic path $\gamma_{u_1,u_2}:[0,1] \to UT_{0}\mathbb{H}^{d+1}$ on the sphere with constant speed connecting $u_1$ to $u_2$. Hence we define the path $[z_1,z_2]^*$ as
\begin{eqnarray}
\label{E:defofstar}
[0,1] &\to& \mathbb{H}^{d+1} \nonumber\\
t &\mapsto& \big((1-t)r_1+tr_2;\gamma_{u_1,u_2}(\phi_{r_1,r_2,\widehat{u_1,u_2}}(t))\big),
\end{eqnarray}
where $\phi_{r_1,r_2,\widehat{u_1,u_2}}:[0,1] \to [0,1]$ is defined as:
\begin{eqnarray*}
\phi_{r_1,r_2,\widehat{u_1,u_2}}(t):=\frac{1}{\widehat{u_1,u_2}}\arccos \left(\frac{(1-t)\sinh(r_1)+t\cos(\widehat{u_1 u_2})\sinh(r_2)}{\sinh((1-t)r_1+tr_2)}\right).
\end{eqnarray*}
This function $\phi_{r_1,r_2,\widehat{u_1,u_2}}$ is built to ensure that the distance to the origin $z_1$ is monotone along the path $[z_1,z_2]^*$. Indeed, by the hyperbolic law of cosines (\ref{E:cosinelaw}),
\begin{eqnarray*}
&&\cosh d\big(z_1,((1-t)r_1+tr_2;\gamma_{u_1,u_2}(\phi(t))\big) \\
&&=\cosh(r_1)\cosh((1-t)r_1+tr_2)-\cos(\phi(t)(\widehat{u_1,u_2}))\sinh(r_1)\sinh((1-t)r_1+tr_2) \\
&&=t\left[ \cosh(r_1)\cosh(r_2)-\cos(\widehat{u_1,u_2})\sinh(r_1)\sinh(r_2)\right] 
\end{eqnarray*}
is monotone in $t$.

We define $[z_1,z_2[^*:=[z_1,z_2]^* \backslash \{z_2\}$ and $]z_1,z_2]^*:=[z_1,z_2]^* \backslash \{z_1\}$. It is possible to assume that $\mathcal{N}$ does not contain two points $z_1,z_2$ such that $0 \in [z_1,z_2]$ since this event has probability $1$. Let us now define the random set $\RST$ by connecting each point $z \in \mathcal{N}$ to $A(z)$ by the path $[z,A(z)]^*$:
$$\RST:=\bigsqcup_{z \in \mathcal{N}} [z,A(z)[^*.$$
It may exist some points $z$ belonging to several paths $[z_1,A(z_1)[^*$, ..., $[z_k,A(z_k)[^*$; in that case, $z$ is counted with multiplicity $k$ in $\RST$.  Formally, we should define the $\RST$ as $\bigcup_{z \in \mathcal{N}} [z,A(z)[^* \times \{z\} \subset \mathbb{H}^{d+1} \times \mathbb{H}^{d+1}.$, i.e. an element $(z,z') \in \RST$ is a couple where $z \in \mathbb{H}^{d+1}$ is a point of the RST and $z'$ is the root of an edge containing $z$. For $(z,z') \in \RST$, we define
\begin{equation}\label{def:zfleche}
z_\downarrow=z', \quad z_\uparrow=A(z').
\end{equation}
For the sake of simplification, we will commit an abuse of notations by considering that $\RST \subset \mathbb{H}^{d+1}$ and identifying an element $(z,z') \in \RST$ to the corresponding point $z \in \mathbb{H}^{d+1}$. Given $z \in \RST$, let $n:=\min \{k \ge 0,~A^{(k)}(z_\uparrow)=0\}$ be the number of steps required to reach the origin from $z_\uparrow$; we define the \emph{trajectory} from $z$ as
\begin{eqnarray*}
\pi(z):=[z,z_\uparrow]^* \cup \bigcup_{k=0}^{n-1} \left[A^{(k)}(z_\uparrow),A^{(k+1)}(z_\uparrow)\right]^*.
\end{eqnarray*}

For $r>0$, we define the \emph{level} $r$ as $$\mathcal{L}_r:=\RST \cap S(r).$$ For $0<r \le r'$ and for $z' \in \mathcal{L}_{r'}$, the \emph{ancestor at level} $r$ of $z'$, denoted by $\mathcal{A}_r^{r'}(z')$ is the intersection point of $\pi(z')$ and $S(r)$. For $0<r \le r'$ and for $z \in \mathcal{L}_{r}$, the \tvc{\emph{set of descendants at level}} $r'$ is defined as 
\[\mathcal{D}_r^{r'}(z):=\{z' \in \mathcal{L}_{r'}, z \in \pi(z')\}\]
(we extend the notation for $z \notin \mathcal{L}_{r}$ by setting $\mathcal{D}_r^{r'}(z):=\emptyset$).
For $z=(r;u) \in \RST$, the \tvc{\emph{descendant subtree}} of $z$ is defined as $\mathcal{D}(z):=\bigcup_{r' \ge r} \mathcal{D}_r^{r'}(z)$, \tvc{see Figure \ref{Fig:ancestor}.} Recall that an \emph{infinite backward path} is a sequence $(z_n)_{n \in \mathbb{N}} \in \left(\mathbb{H}^{d+1}\right)^{\mathbb{N}}$ such that $z_0=0$ and $z_n=A(z_{n+1})$ for all $n \ge 0$.

\begin{figure}[!ht]
    \centering
\begin{tikzpicture}[scale=1.5,decoration={brace}]

\node at (0.78,0) {\includegraphics[scale=0.75]{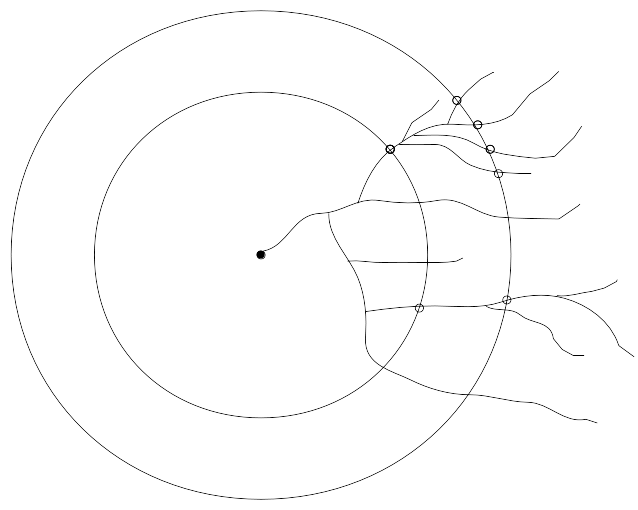}};
\draw (0,0) node[below right] {$0$};


\draw[<->] (0,0)--(-1.55,1.55);
\draw (-0.775,0.775) node[below left] {$r$};
\draw[<->] (0,0)--(-3.3,0);
\draw (-1.65,0) node[below] {$r'$};

\draw (1.7,1.5) node[below right] {$z_2 \in \mathcal{L}_r$};
\draw (3.2,-0.6) node[above right] {$z_1 \in \mathcal{L}_{r'}$};
\draw (2,-0.5) node[below right] {$A_r^{r'}(z_1)$};

\draw[decorate] (2.76,2.2)--(3.38,1.05);
\draw (3.07,1.625) node[right] {$\mathcal{D}_r^{r'}(z_2)$};

\end{tikzpicture}

    \caption{Representation of levels $r$ and $r'$, the ancestor $A_r^{r'}(\cdot)$ and the set of descendants $\mathcal{D}_r^{r'}(\cdot)$. \tvc{Note that ancestors are towards the root $0$ of the RST, which is the ``forward'' direction. Descendants are seen when moving to higher radii, which is the ``backward direction''.}}
    \label{Fig:ancestor}
\end{figure}
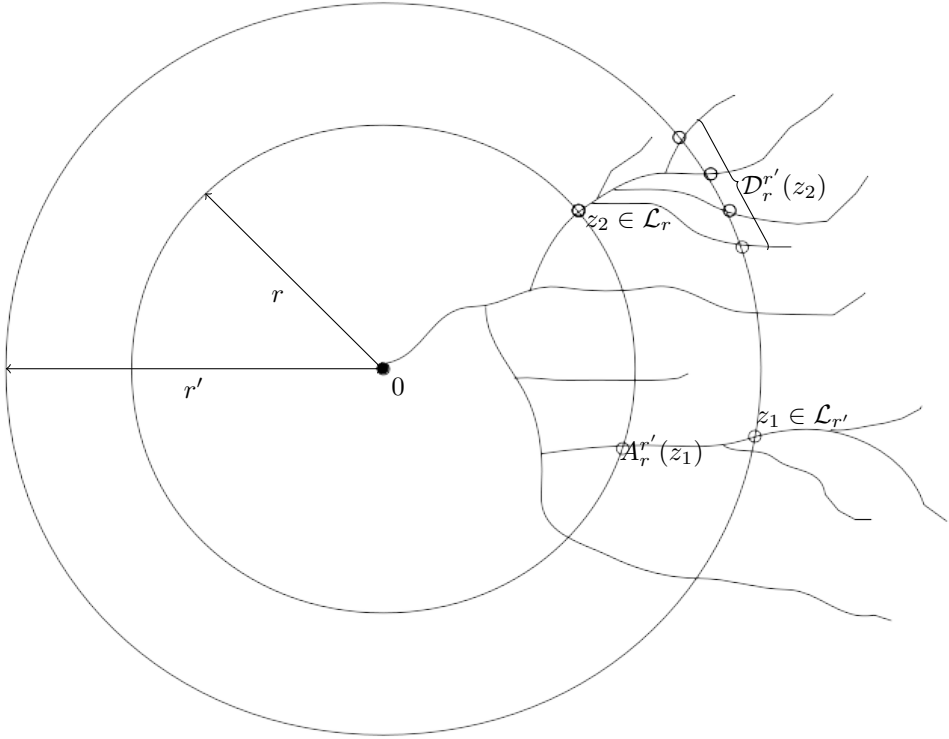

\subsection{Sketch of proofs}

In order to prove our main result (Theorem \ref{Thm:mainthm}), the key point is to upper-bound angular deviations of trajectories. We first introduce two quantities, the \emph{Cumulative Forward angular Deviations (CFD)} and \emph{Maximal Backward Deviations (MBD)} to quantify those fluctuations.

\begin{definition}[Cumulative Forward angular Deviations]
Let $0<r \le r'$ and $z' \in S(r')$. If $z' \notin \RST$, we set $\CFD_r^{r'}(z')=0$ by convention and we now suppose that $z' \in \RST$. Let $z:=\mathcal{A}_r^{r'}(z')$. We define the \emph{Cumulative Forward angular Deviations} of $z'$ between levels $r$ and $r'$ as
\begin{eqnarray*}
\CFD_r^{r'}(z'):=
\left\{
\begin{aligned}[l]
&\widehat{z' 0 z} \text{ if } z_\downarrow=z'_\downarrow,\\
&\widehat{z' 0 z'_\uparrow}+\sum_{k=0}^{n-1}\widehat{A^{(k)}(z'_\uparrow) 0 A^{(k+1)}(z'_\uparrow)}+\widehat{z_\downarrow 0 z} \text{ else,}
\end{aligned}
\right.
\end{eqnarray*}
where $n$ is the unique non negative integer such that $A^{(n)}(z'_\uparrow)=z_\downarrow$.
\end{definition}

\begin{definition}[Maximal Backward angular Deviations]
\label{def:defmbd}
Let $0<r \le r'$ and $z \in S(r)$. We define the \emph{Maximal Backward angular Deviations} between levels $r$ and $r'$ as
\begin{eqnarray*}
\MBD_r^{r'}(z):=
\left\{
\begin{aligned}[l]
&0 \text{ if } z \notin \RST,\\
&\sup_{r'' \in [r,r']}\max_{z'' \in \mathcal{D}_{r}^{r''}(z)} \CFD_{r}^{r''}(z'') \text{ if } z \in \RST.
\end{aligned}
\right.
\end{eqnarray*}
We extend the definition to $r'=\infty$ by setting:
\begin{eqnarray*}
\MBD_r^\infty(z):=\lim_{r' \to \infty}\MBD_r^{r'}(z),
\end{eqnarray*}
the limit exists since $r' \mapsto \MBD_r^{r'}(z)$ is non-decreasing.
\end{definition}

\tvc{All forward paths of the RST end at the root $0$, but they can be finite in the backward direction if a vertex is the ancestor of no other vertex. The quantity $\MBD_{r}^{r'}(z)$ takes into account backward paths from $z$ that either end (in the backward direction) before level $r'$ or reach level $r'$.
}

These quantities will be upper-bounded in two steps. First, a percolation argument is used to control angular deviations in any annulus of width $\delta>0$ for some small $\delta>0$ (Proposition \ref{Prop:controlannulus}) and then we deduce a global control of angular deviations (Proposition \ref{Prop:globalfluct}). These two propositions are proved in Sections \ref{S:proofcontrolannulus} and \ref{S:proofglobalfluct}.

\tvc{Let us introduce some further notations. For $r>0$, $z \in \mathbb{H}^{d+1}$ and $\theta>0$, we define
\begin{eqnarray}
\label{E:defBSR}
B_{S(r)}(z,\theta):=\Cone(z,\theta) \cap S(r).
\end{eqnarray}
}
\dave{We will control moments of the maximal backward angular deviations between close radii $r$ and $r+\delta$ when $z$ belongs to such arc $B_{S(r)}(z,\theta)$.}

\begin{proposition}
\label{Prop:controlannulus}
There exists $\delta>0$ such that, for any $p \ge 1$, there exists $C=C(d,p)>0$ such that for any $r>0$, $\theta \ge 0$ and any direction $u \in \mathbb{S}^d$,
\begin{eqnarray}
\mathbb{E}\left[ \sum_{z \in B_{S(r)}(u,\theta) \cap \RST} \left(\MBD_r^{r+\delta}(z)\right)^p \right] \le C \theta^d e^{r(d-p)},
\end{eqnarray}
\end{proposition}

\begin{proposition}
\label{Prop:globalfluct}
For any $p\geq 3d/2$, there exists some constant $C_{\fluct}>0$ such that, for any $0<r_0<\infty$, $A>0$ and any direction $u \in \mathbb{S}^d$,
\begin{eqnarray*}
\mathbb{E}\left[\sum_{z \in B_{S(r_0)} (u,Ae^{-r_0}) \cap \RST} \left( \MBD_{r_0}^\infty(z) \right)^p \right] \le  C_{\fluct}A^de^{-r_0p}.
\end{eqnarray*}
\end{proposition}

These controls of angular deviations will be first used to show that the RST is \emph{straight} (Proposition \ref{Prop:straightness}). The straightness property is the key to show (i), (ii) and (iv) in Theorem \ref{Thm:mainthm}.

\begin{proposition}[straightness property]
\label{Prop:straightness}
Almost surely, the following happens. For any $\varepsilon>0$, there exists some $R_0>0$, such that, for any radius $r_0 \ge R_0$, for any $z \in \RST$ with $d(0,z) \ge r_0$, the \tvc{descendant} subtree $\mathcal{D}(z)$ is contained in a cone of apex $0$ and aperture $e^{-(1-\varepsilon)r_0}$, i.e. for any $z',z'' \in \mathcal{D}(z)$, $\widehat{z'0z''} \le e^{-(1-\varepsilon)r_0}$.
\end{proposition}

The proof of (iii) in Theorem \ref{Thm:mainthm} exploits the controls of angular  deviations (Proposition \ref{Prop:globalfluct}) and the link existing between the RST and the hyperbolic Directed Spanning Forest introduced: the DSF approximates locally the RST far from the origin. The uniqueness of the infinite backward path with some given deterministic asymptotic direction has been shown for the DSF \cite{HyperbolicDSF}, and the local coupling existing between the two models permits to show that this property remains true for the RST.

\section{Proof of Theorem \ref{Thm:mainthm}}
\label{S:proofmainthm}

Here we assume that Propositions \ref{Prop:globalfluct} and \ref{Prop:straightness} are proved and we show that it implies Theorem \ref{Thm:mainthm}.

\subsection{The existence part: proof of (i),(ii),(iv) and (v)}
\label{S:existencepart}

\paragraph{Proof of (i) and (ii)} We first show that any infinite backward path admits an asymptotic direction and that any ideal boundary point is the asymptotic direction of an infinite backward path. The strategy consists in exploiting the straightness property (Proposition \ref{Prop:straightness}):

Let $(z_n)_{n\geq 0}$ be an infinite backward path, we prove that $(z_n)_{n\geq 0}$ admits an asymptotic direction. For $n \ge 0$, let us decompose $z_n$ in polar coordinates: $z_n=(r_n;u_n)$. Proposition \ref{Prop:straightness} immediately implies that the sequence $(u_n)_{n \ge 0}$ is a Cauchy sequence in $UT_{0}\mathbb{H}^{d+1} \simeq \mathbb{S}^{d}$.
\tvc{To see this, let $\varepsilon>0$. Since the path $(z_n)$ is infinite, the sequence $(r_n)$ converges to infinity. Let $n_\varepsilon$ be such that $e^{-r_{n_\varepsilon}/2}\leq \varepsilon$. By the straightness property, there exists $N_0$ such that for $n_0 \geq N_0$, the path $(z_n)_{n\geq n_0}$ remains inside $\Cone(z_{n_0},e^{-r_{n_0}/2})\cap B(r_{n_0})^c$. Thus, for $n_0\geq N_0\vee n_\varepsilon$, the path $(z_n)_{n\geq n_0}$ remains in a cone of aperture $\varepsilon$, proving that 
$(u_n)_{n\geq 0}$ is a Cauchy sequence.}
Thus the sequence $(u_n)_{n \ge 0}$ converges, and so $(z_n)_{n\geq 0}$ converges to some boundary point $z_\infty \in \partial\mathbb{H}^{d+1}$, \tvc{which proves (i).}\\

Let $\Psi=\{\lim_{n \to \infty} z_n,~(z_n) \text{ is an infinite backward path}\} \subset \partial \mathbb{H}^{d+1}$ be the set of asymptotic directions reached by at least one infinite backward path. In order to prove that $\Psi=\partial \mathbb{H}^{d+1}$, we proceed in two steps: we first show that $\Psi$ is dense in $\partial\mathbb{H}^{d+1}$, then we show that $\Psi$ is closed in $\partial\mathbb{H}^{d+1}$.\\
Since the $\RST$ is an infinite tree with finite degree a.s. (Proposition \ref{Prop:firstprop}), there exists an infinite backward path from $0$ and the corresponding infinite backward path converges to an ideal boundary point by the previous paragraph, thus $\Psi \neq \emptyset$ almost surely.

We denote by $\Stab(0)$ the set of isometries that fix $0$, in particular it contains rotations centred at $0$. Let $B$ be an open subset of $\partial \mathbb{H}^{d+1}$. Since $\partial \mathbb{H}^{d+1} \simeq \mathbb{S}^d$ is compact, there exists finitely many isometries $\gamma_1,...,\gamma_k \in \Stab(0)$ such that $\bigcup_{i=1,...,k} \gamma_iB=\partial \mathbb{H}^{d+1}$. The random set $\RST$ is invariant in distribution by $\Stab(0)$, so the events $\{\Psi \cap \gamma_i B \neq \emptyset\}$ all have the same probability. Since $\Psi \neq \emptyset$ almost surely, $\mathbb{P}\left[\bigcup_{i=1,...,k} \{\Psi \cap \gamma_i B \neq \emptyset\} \right]=1$ therefore $\mathbb{P}(\Psi \cap B \neq \emptyset)>0$. In addition, for any neighbourhood $\Phi \subset \overline{\mathbb{H}^{d+1}}$ of $B$, the event $\{\Psi \cap B \neq \emptyset\}$ is entirely determined by $\mathcal{N} \cap \Phi$, therefore it has probability $0$ or $1$, \tvc{by Kolmogorov's 0-1 law.} Thus $\Psi \cap B \neq \emptyset$ almost surely, \tvc{since we already showed that $\mathbb{P}(\Psi \cap B \neq \emptyset)>0$.} Since the topology on $\partial \mathbb{H}^{d+1}$ admits a countable basis, $\Psi$ is almost surely dense in $\partial\mathbb{H}^{d+1}$.

It remains to show that $\Psi$ is a closed subset of $\partial\mathbb{H}^{d+1}$. Let $I \in \overline{\Psi}$ (recall that $\overline{\Psi}$ is the closure of $\Psi$ in $\overline{\mathbb{H}^{d+1}}$). We construct by induction an infinite backward path $(z_n)_{n \ge 0} \in \mathcal{N}^{\mathbb{N}^*}$ such that, for any $i \in \mathbb{N}$, $I \in \overline{\mathcal{D}(z_i)}$. Suppose $0,...,z_{i-1}$ already defined such that $z_j=A(z_{j+1})$ for $0 \le j \le i-2$ and $I \in \overline{\mathcal D(z_{i-1})}$. Since the vertex $z_{i-1}$ has finitely many daughters, there exists some $z \in A^{(-1)}(z_{i-1})$ such that $I \in \overline{\mathcal{D}(z)}$. We define $z_i$ as such a $z$.

We now use straightness to show that the infinite backward path $(z_n)$ constructed above converges to $I$ (and thus $I \in \Psi$). This infinite backward path converges to some $I' \in \partial\mathbb{H}^{d+1}$ by (i). Let $\varepsilon>0$, by Proposition \ref{Prop:straightness} there exists some $i \ge 0$ such that $\mathcal{D}(z_i)$ (and thus $\overline{\mathcal{D}(z_i)}$) is contained in a cone of apex $0$ and aperture at most $\varepsilon$. Since both $I$ and $I'$ belong to $\overline{\mathcal{D}(z_i)}$, $\widehat{I 0 I'} \le \varepsilon$. Thus $I=I'$, \tvc{which achieves the proof of (ii).}

\paragraph{Proof of (iv) and (v)} Let us denote by $\Psi' \subset \partial \mathbb{H}^{d+1}$ the set of asymptotic directions with two infinite backward paths. To show (iv), we first show that, a.s. $\Psi'\neq \emptyset$. For $z \in \RST$, let us define $\Psi_z \subset \partial \mathbb{H}^{d+1}$ as the set of asymptotic directions of infinite backward paths from $z$ (empty or not). By the same argument as in Step 2, $\Psi_z$ is a closed subset of $\partial \mathbb{H}^{d+1}$. By (ii), a.s., there exists at least two infinite backward paths, so there exists a.s. some level $r_0>0$ with two points connected to infinity. Thus $\{\Psi_z,~z \in \mathcal{L}(r_0)\}$ is a covering of $\partial \mathbb{H}^{d+1}$ by closed subsets, where at least two of them are nonempty. Since $\partial \mathbb{H}^{d+1}$ is connected, it implies that there exists $z_1,z_2 \in \mathcal{L}_{r_0}$ such that $\Psi_{z_1} \cap \Psi_{z_2} \neq \emptyset$. Thus $\Psi' \neq \emptyset$ a.s.

We use the same argument as in Step 2 to deduce that $\Psi'$ is dense. Let $B$ be an open subset of $\partial \mathbb{H}^{d+1}$. Since $\partial \mathbb{H}^{d+1} \simeq \mathbb{S}^d$ is compact, there exist finitely many isometries $\gamma_1,...,\gamma_k \in \Stab(0)$ such that $\bigcup_{i=1,...,k} \gamma_iB=\partial \mathbb{H}^{d+1}$. The random set $\RST$ is invariant in distribution by $\Stab(0)$, so the events $ \Gamma_i:=\{\Psi' \cap \gamma_iB \neq \emptyset\}$
all have the same probability. Since $\Psi \neq \emptyset$ almost surely, $\mathbb{P}\left[\bigcup_{i=1,...,k} \Gamma_i \right]=1$ and therefore $\mathbb{P}(\Gamma_i)>0$. In addition, for any neighbourhood $\Phi \subset \overline{\mathbb{H}^{d+1}}$ of $B$, the event $\Gamma_i$ is entirely determined by $\mathcal{N} \cap \Phi$, which implies that it has probability $0$ or $1$. Thus $\Psi' \cap B \neq \emptyset$ almost surely. Since the topology on $\partial \mathbb{H}^{d+1}$ admits a countable basis, $\Psi'$ is almost surely dense in $\partial\mathbb{H}^{d+1}$.

The proof of (v) is done by exploiting the planarity in the bi-dimensional case (Proposition \ref{Prop:firstprop}). Let us associate to any $z_\infty \in \Psi'$ a couple of vertices $P(z_\infty)=(z_1,z_2) \in \mathcal{N}^2$ with $z_1 \neq z_2$ such that $z_\infty \in \Psi_{z_1} \cap \Psi_{z_2}$. By planarity, such an application $P$ must be injective. Indeed, if $z_\infty \neq z'_\infty$ are such that $P(z_\infty)=P(z'_\infty)=(z_1,z_2)$, then
there exists four distinct backward infinite paths joining $z_\infty$ to $z_1$, $z_\infty$ to $z_2$, $z'_\infty$ to $z_1$ and $z'_\infty$ to $z_2$. This implies that two paths among them intersect each other, even if the representation of edges are replaced by geodesics, which contradicts planarity. Therefore $\Psi'$ is a.s. countable in the case $d=1$.

\subsection{The uniqueness part: proof of (iii)}

The strategy is to exploit the link between the hyperbolic RST and the hyperbolic Directed Spanning Forest (DSF). 
\dave{First, recall that the space $\mathbb{H}^{d+1}$ can be described with several isometric models. We have already seen the open-ball model $(I,ds_I^2)$ and we introduce here another one, called the upper half-space model, $(H, ds^2)$, where:
\begin{eqnarray*}
H=\{z=(x_1,\dots,x_d,y) \in \mathbb{R}^{d+1},~y>0\}, \quad ds^2=\frac{dx_1^2+\dots+dx_d^2+dy^2}{y^2}.
\end{eqnarray*}
In the following, we will identify the point $z=(x_1,\dots,x_{d},y) \in H$ with the couple $(x,y) \in \mathbb{R}^d \times \mathbb{R}_+^*$ with
$x:=(x_1,\dots ,x_d)$ and $y:=x_{d+1}$.
The coordinate $x$ is referred as the \emph{abscissa} and $y$ as the \emph{ordinate}. The latter coordinate $y$ plays a special role and will be the direction in which the DSF is built. 
}

Let us remind that, in the half-space representation, the boundary set $\partial \mathbb{H}^{d+1}$ is identified as the boundary hyper-plane $\mathbb{R}^d \times \{0\}$, plus an additional point at infinity denoted by $\infty$, obtained by compactifying the closed half-space $\mathbb{R}^d \times \mathbb{R}_+$. Let us define $I_\infty$ as the boundary point represented by $(0,0)$ in $H$.

\tvc{We now recall the definition of the hyperbolic DSF with direction $\infty$ (the point defined in the preceding paragraph) in the half-space model $(H,ds^2)$. The hyperbolic DSF with direction $\infty$ is a graph with set vertices the points of our PPP $\mathcal{N}$ and whose nodes have out-degree 1. We connect each point $z \in \mathcal{N}$ to the closest Poisson point in the direction of $\infty$: this point, denoted as $A_{\DSF}(z)$, is the parent of $z=(x,y)\in \mathcal{N}$ and is defined as:
\begin{equation}
\label{def:ADSF}
A_{\DSF}(z):=\argmin_{z'=(x',y') \in \mathcal{N},~y'>y} d(z',z).
\end{equation}
where $d(.,.)$ is the hyperbolic distance. Underlying this notion is the notion of horodistance (see \cite{HyperbolicDSF}), but in the case where the direction is $\infty$, the definition boils down to \eqref{def:ADSF}.
 }

\tvc{
For any $h \ge 0$, let us now define $\RST(h)$ as the Radial Spanning Tree of $\mathcal{N}$ with origin $O(h):=(0,e^h)$ similarly as in Definition \ref{Def:rsthyp}. It is the tree rooted at $O(h)$ with vertex set $\mathcal{N}$ and whose nodes have out-degree 1. Each point $z\in \mathcal{N}$ is connected to its parent 
\begin{equation}
\label{def:ARST}
A_{\RST}(z)=\argmin_{z'\in (\mathcal{N}\cup\{O(h)\})\cap B(O(h),r)}d(z',z),\qquad \mbox{ with }\quad r=d(O(h),z).
\end{equation}
}
\tvc{To avoid confusion, the parent of $z$ in the hyperbolic RST is now denoted by $A_{\RST}(z)$ instead of $A(z)$ (see Definition \ref{Def:rsthyp}) when necessary.
}

We will also consider, for any given $h \ge 0$, the direction toward $I_\infty$ defined as $u:=(0,...,-1) \in UT_{O(h)}\mathbb{H}^{d+1} \simeq \mathbb{S}^d$. The proof is based on the two following propositions. The next one asserts that the $\RST(h)$ and the $\DSF$ coincide in a given compact set when $h$ is large enough.

\begin{proposition}[Coupling between $\RST$ and $\DSF$]
\label{Prop:coupling}
Let $K \subset \mathbb H^{d+1}$ be some compact set. Then
\begin{eqnarray*}
\lim_{h \to \infty} \mathbb{P}[\forall z \in \mathcal{N} \cap K,~A_{\RST(h)}(z)=A_{\DSF}(z)]=1.
\end{eqnarray*}
\end{proposition}
For $A,a,h \ge 0$, let us define:
\begin{eqnarray*}
&&\Vois(A,h):=\Cone_{O(h)}(I_\infty,Ae^{-h}) \backslash B(O(h),h), \\
&&\Vois'(A,a,h):=(B(O(h),h+a) \cap \Cone_{O(h)}(I_\infty,Ae^{-h})) \backslash B(O(h),h),\\
&&\Vois''(A,a,h):=\Cone_{O(h)}(I_\infty,Ae^{-h-a}) \backslash B(O(h),h+a),
\end{eqnarray*}
where $\Cone_{z_0}(z,\theta)$ denotes the cone with apex $z_0$, direction $z$ and aperture $\theta$. Let us also define:
\begin{eqnarray*}
&&\Cyl(A):=B_{\mathbb{R}^d}\left(0,A\right) \times \left(0,\frac{3}{2}\right], \\
&&\Cyl'(A,a):=B_{\mathbb{R}^d}\left(0,A\right) \times \left[\frac{1}{2}e^{-a},\frac{3}{2}\right], \\
&&\Cyl''(A,a):=B_{\mathbb{R}^d}\left(0,Ae^{-a}\right) \times \left(0,\frac{3}{2}e^{-a}\right].
\end{eqnarray*}
The sets $\Vois(A,h), \Vois'(A,a,h), \Vois''(A,a,h)$ and $\Cyl(A), \Cyl'(A,a), \Cyl''(A,a)$ are represented in Figure \ref{Fig:figcouplage}. We will use the following geometrical fact:
\begin{lemma}

\label{Lem:geomlemma}
For any $A,a \ge 0$, $h$ can be chosen large enough such that
\begin{eqnarray}
\label{E:inclusionsets}
\Vois(A,h) \subset \Cyl(A), \quad \Vois'(A,a,h) \subset \Cyl'(A,a), \text{ and }\Vois''(A,a,h) \subset \Cyl''(A,a).
\end{eqnarray}
\end{lemma}

\dave{Let us now prove (iii) of Theorem \ref{Thm:mainthm}. }For $A,h \ge 0$, let us define the event $E(A,h)$ saying that every infinite backward path converging to $I_\infty$ in $\RST(h)$ restricted to the annulus $\mathbb{H}^{d+1} \backslash \tvc{B(O(h),h)} $ are contained in $\Vois(A,h)$:
\begin{eqnarray*}
E(A,h):=\{\forall z \in \RST(h) \cap (\mathbb{H}^{d+1} \backslash \tvc{B(O(h),h)}),~I_\infty \in \overline{\mathcal{D}_{\RST(h)}(z)} \implies z \in \Vois(A,h)\},
\end{eqnarray*}
where $\mathcal{D}_{\RST(h)}(z)$ denotes the \tvc{descendant} subtree of $z$ in $\RST(h)$. Let us also define, for $A,h,a\geq 0$, the event $E'(A,h,a)$ saying that every infinite backward path converging to $I_\infty$ in $\RST(h)$ restricted to the annulus $\mathbb{H}^{d+1} \backslash B(O(h),h+a)$ are contained in $\Vois''(A,a,h)$:
\begin{eqnarray*}
E'(A,a,h):=\{\forall z \in \RST(h) \cap (\mathbb{H}^{d+1} \backslash \tvc{B(O(h),h)}),~I_\infty \in \overline{\mathcal{D}_{\RST(h)}(z)} \implies z \in \Vois''(A,a,h)\}.
\end{eqnarray*}
The following proposition asserts that, uniformly in $h$, the events $E(A,h)$ and \tvc{$E'(A,a,h)$} occur with high probability when $A$ is large.

\begin{proposition}
\label{Prop:cylinder}
We have
\begin{eqnarray*}
\lim_{A \to \infty} \liminf_{h \to \infty} \mathbb{P}[E(A,h)]=1, \quad \lim_{A \to \infty} \liminf_{h \to \infty} \mathbb{P}[E'(A,a,h)]=1 \text{ for any } a \ge 0.
\end{eqnarray*}
\end{proposition}

Let us assume Propositions \ref{Prop:coupling}, \ref{Prop:cylinder} and Lemma \ref{Lem:geomlemma} for the moment and let us prove part (iii) of Theorem \ref{Thm:mainthm}. Let $h \ge 0$. \tvc{By Point (i) of Theorem \ref{Thm:mainthm}, there exists a.s. an infinite backward path converging to $I_\infty$ in $\RST(h)$.} Let us define the event
\begin{eqnarray*}
U(h):=\{\text{there is a unique infinite backward path converging to $I_\infty$ in $\RST(h)$}\}.
\end{eqnarray*}By \tvc{isometric} invariance, $\mathbb{P}[U(h)]$ is independent of $h$. Let us suppose for contradiction that $q:=\mathbb{P}[U(h)^c]>0$. For $A,a,h \ge 0$, let us define the event
\begin{eqnarray*}
\CO(A,a,h):=\{\forall z \in \mathcal{N} \cap \Cyl'(A,a),~A_{\RST(h)}(z)=A_\DSF(z)\}.
\end{eqnarray*}
By Proposition \ref{Prop:cylinder}, $A\dave{=A(a)}$ can be chosen such that
\begin{eqnarray*}
\liminf_{h \to \infty} \mathbb{P}[E(A,h)]>1-q/\dave{8}, \quad \mbox{ and }\quad \liminf_{h\to \infty} \mathbb{P}[E'(A,a,h)]>1-q/\dave{8}.
\end{eqnarray*}
Then, by Proposition \ref{Prop:coupling} applied to the compact set $K:=\Cyl'(A,a)$ and Lemma \ref{Lem:geomlemma}, $h\dave{=h(a)}$ can be chosen large enough such that inclusions (\ref{E:inclusionsets}) hold and such that
\begin{eqnarray*}
\mathbb{P}[\CO(A,a,h)]>1-q/4, \quad \mathbb{P}[E(A,h)] \ge 1-q/4 \text{ and } \mathbb{P}[E'(A,a,h)] \ge 1-q/4.
\end{eqnarray*}

Let us define the event $Z(a)$ as
\begin{eqnarray*}
Z(a):=U(h\dave{(a)})^c \cap E(A\dave{(a)},h\dave{(a)}) \cap E'(A\dave{(a)},a,h\dave{(a)}) \cap \CO(A\dave{(a)},a,h\dave{(a)}),
\end{eqnarray*}
and define
\begin{eqnarray*}
Z:=\bigcap_{a_0>0} \bigcup_{a \ge a_0} Z(a).
\end{eqnarray*}
For the choices of $a$ done before, $\mathbb{P}[Z(a)] \ge q/4$ and so $\mathbb{P}[Z] \ge q/4>0$. On the event $Z(a)$, and because inclusion (\ref{E:inclusionsets}) holds, there exists two infinite backward paths in $\RST(h)$ whose restrictions to $\mathbb{R}^d \times (0,3/2]$ are contained in $\Cyl(A,a)$ and converging to $I_\infty$, and intersecting $\Cyl''(A,a)$. These two infinite backward paths coincide with those of $\DSF$ inside $\Cyl'(A,a)$. Thus, in $\DSF$, there exist two infinite backward paths contained in $\Cyl(A,a)$ and intersecting $\Cyl''(A,a)$. Therefore, since this is true for all $a>0$ \dave{sufficiently large}, on the event $Z$, it is possible to construct two infinite backward paths converging to $I_\infty$ in $\DSF$ using the fact that $\DSF$ is locally finite (it is true by \cite[Proposition 2.10]{HyperbolicDSF}). However, by \cite[Theorem 1.2]{HyperbolicDSF}, there exists almost surely a unique infinite backward path converging to $I_\infty$ in $\DSF$. This leads to a contradiction, which achieves the proof of (iii).

\begin{figure}[!h]
\centering
\caption{Representation of the sets $\Vois(A,h)$, $\Vois'(A,a,h)$, $\Vois''(A,a,h)$ and $\Cyl(A)$, $\Cyl'(A,a)$, $\Cyl''(A,a)$. The backward paths of $\RST(h)$ converging to $0$ (in \textcolor{blue}{blue}) are all contained in $\Vois''(A,a,h)$ up to level $h+a$ and contained in $\Vois(A,a,h)$ up to level $h$. In the dashed area ($\Cyl'(A,a)$), the $\DSF$ and $\RST(h)$ coincide.}
\label{Fig:figcouplage}
\begin{tikzpicture}[scale=3,decoration={brace}]

\draw (0,0) node[below] {$0$};
\draw[->] (-2.5,0)--(2.5,0) node[below right] {$x \in \mathbb{R}^d$};
\draw[->] (0,0)--(0,3) node[above left]{$y \in \mathbb{R}$};

\draw (0,2.718) arc (137.844:180:4.05);
\draw (0,2.718) arc (164.492:180:10.166);
\draw (0,2.718) arc (15.508:0:10.166);
\draw (0,2.718) arc (42.156:0:4.05);
\draw[dashed] (0.888,1.126) arc (-73.86:-106.14:3.195);
\draw[dashed] (1.026,0.421) arc (-84.028:-95.972:9.859);
\draw (0,2.718) node[above right] {O(h)};

\draw (-2,1.5) -- (2,1.5) -- (2,0) -- (-2,0) -- (-2,1.5);
\draw (-2,0.184) -- (2,0.184);
\draw (-0.736,0.552) -- (0.736,0.552) -- (0.736,0) -- (-0.736,0) -- (-0.736,0.552);
\draw[decorate] (0.746,0.552) -- (0.746,0);
\draw (0.746,0.276) node[right]{$\Cyl''(A,a)$};
\draw[decorate] (2.01,1.5) -- (2.01,0.184);
\draw (2.01,0.842) node[right]{$\Cyl'(A,a)$};
\draw[decorate] (-2.01,0) -- (-2.01,1.5) ;
\draw (-2.01,0.75) node[left]{$\Cyl(A)$};

\draw[decorate] (0.908,1.126) -- (1.046,0.421);
\draw (0.977,0.784) node[right]{$\Vois'(A,a,h)$};
\draw[decorate] (-1.088,0.0) -- (-0.928,1.126);
\draw (-1.008,0.563) node[left]{$\Vois(A,h)$};
\draw[decorate]  (-0.38,0.0) -- (-0.373,0.375);
\draw (-0.377,0.188) node[left]{$\Vois''(A,a,h)$};

\draw[blue] (0,0) .. controls (-0.4,0.3) and (-0.5,1.3) .. (-1.5,2);
\draw[blue] (0,0) .. controls (-0.1,0.3) and (-0.3,1.3) .. (-0.5,2);
\draw[blue] (0,0) .. controls (0.2,0.3) and (0.5,1.3) .. (0.5,2);
\draw[blue] (0,0) .. controls (0.3,0.3) and (0.7,1.3) .. (1.5,2);

\draw (-0.888,1.126) node[left]{$S(O(h),h)$};
\draw (1.026,0.421) node[right]{$S(O(h),h+a)$};

\end{tikzpicture}
\end{figure}
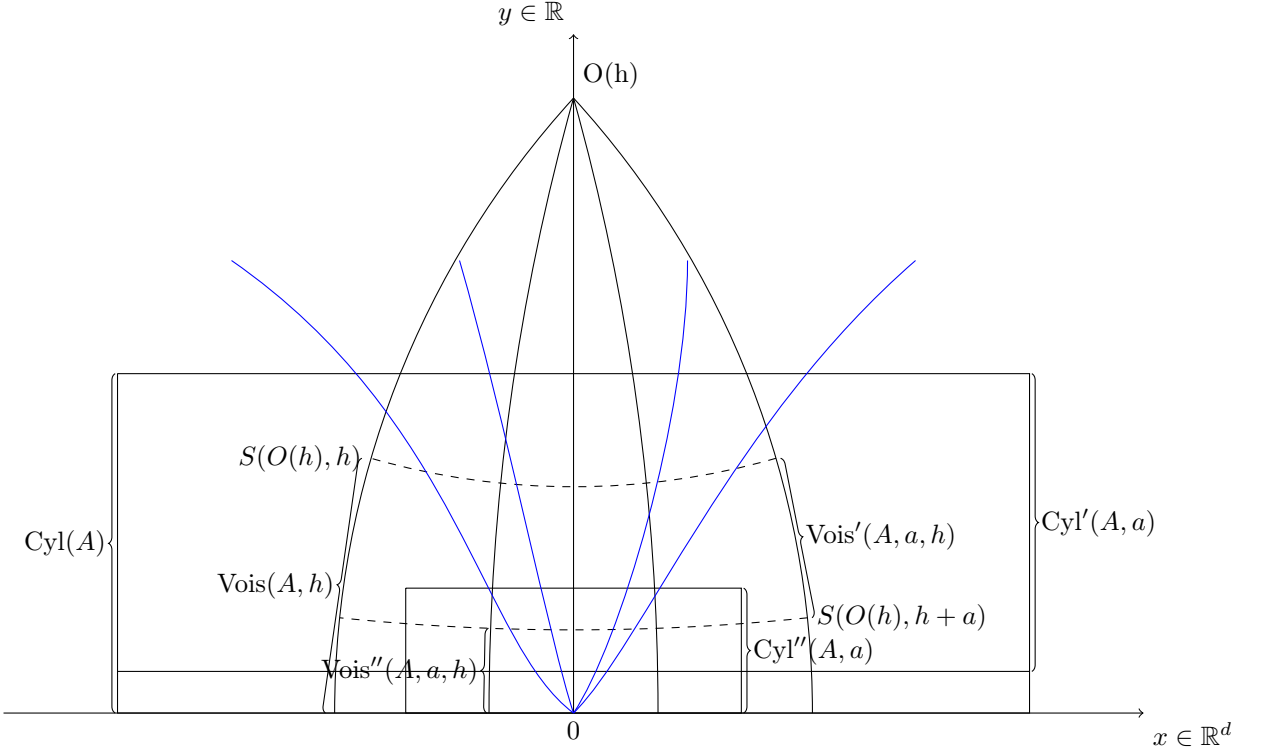

It remains to prove Propositions \ref{Prop:coupling} and \ref{Prop:cylinder}.

\begin{proof}[Proof of Proposition \ref{Prop:coupling}]
Recall the notations \eqref{def:ADSF} and \eqref{def:ARST}. Let us define, for $z=(x,y) \in \mathcal{N}$ and $h \ge 0$:
\tvc{\begin{eqnarray*}
B^+_\DSF(z):= & B(z,d(z,A_{\RST}(z))) \cap (\mathbb{R}^d \times (y,\infty)) , \\
 B^+_{\RST(h)}(z):= & B(z,d(z,A_{\DSF}(z))) \cap B(O(h),d(O(h),z)).
\end{eqnarray*}}
Let $K \subset \mathbb{H}^{d+1}$ be some compact set. For any given $z \in \mathcal{N} \cap K$, $h \ge 0$, $A_\DSF(z)=A_{\RST(h)}(z)$ if and only if
$\mathcal{N} \cap (B^+_\DSF(z) \Delta B^+_{\RST(h)}(z))=\emptyset$.
\tvc{For any $z=(x,y) \in \mathbb H^{d+1}$, because 
$B(O(h),d(O(h),z))$ converges to $\mathbb{R}^d\times (y,\infty)$ when $h\to \infty$, we have that
$\Vol(B^+_\DSF(z) \Delta B^+_{\RST(h)}(z)) \to 0$ (recall that $\Vol$ is the hyperbolic volume).}

\tvc{We now use Campbell formula \cite[Prop. 13.1.IV]{daley-verejones}. Let $z\in \mathbb{H}^{d+1}$ and denote by  $\mathbb{P}_{z'}$ the Palm distribution of $\mathcal{N}$ conditionally on having a point at $z'\in \mathbb{H}^{d+1}$. Then,
\begin{eqnarray*}
&&\mathbb{E}\left[ \#\{z' \in \mathcal{N} \cap B(z,d(z,A(z))),~\mathcal{N} \cap (B^+_\DSF(z) \Delta B^+_{\RST(h)}(z')) \neq \emptyset\}\right] \\
&&=\lambda \int_{B(z,d(z,A(z)))} \mathbb{P}_{z'}\left[ \mathcal{N} \cap (B^+_\DSF(z) \Delta B^+_{\RST(h)}(z')) \neq \emptyset \right]~d\Vol(z') \\
&&=\lambda \int_{B(z,d(z,A(z)))} 1-\exp(-\lambda \Vol(B^+_\DSF(z) \Delta B^+_{\RST(h)}(z'))) ~d\Vol(z') \to 0 \text{ as } h \to \infty
\end{eqnarray*}
by dominated convergence. Proposition \ref{Prop:coupling} follows.}
\end{proof}

\begin{proof}[Proof of Proposition \ref{Prop:cylinder}]
Let $n \in \mathbb{N}$ and $h \ge 0$. Let us define the event
\begin{eqnarray*}
F_n(h):=\{\exists z \in \tvc{S(O(h),h)\cap \Cone_{O(h)}(I_\infty ,2^{n+1}e^{-h})},~\MBD_h^\infty(z)>2^ne^{-h}\}.
\end{eqnarray*}
\tvc{Let us first show that $E(2^n,h)^c \subset \bigcup_{m \ge n} F_m(h)$, which will imply:
\begin{eqnarray}
\mathbb{P}[E(2^n,h)] \ge 1-\sum_{m \ge n}\mathbb{P}[F_m(h)].\label{etape10}
\end{eqnarray}

 If $E(2^n,h)$ does not occur, then there exists some $z' \notin \Vois(2^{n},h)$ such that $I_\infty \in \overline{\mathcal{D}_{\RST(h)}(z')}$. Let $z$ be the ancestor of $z'$ on $S(O(h),h)$. Either $\widehat{zO(h)I_\infty}<2^{n+1} e^{-h}$ and we set $m=n$, or there exists $m > n$ such that $2^m e^{-h} \le \widehat{zO(h)I_\infty}<2^{m+1} e^{-h}$. Then $z \in S(O(h),h)\cap \Cone_{O(h)}(I_\infty ,2^{m+1}e^{-h})$. If $m=n$, $\MBD_h^\infty(z)\geq \widehat{z'O(h)I_\infty}>2^n e^{-h}$ and $F_n(h)$ occurs. If $m>n$, $\MBD_h^\infty(z) \ge \widehat{zO(h)I_\infty} \ge 2^m e^{-h}$ and $F_m(h)$ occurs.} This proves \eqref{etape10}.
\tvc{
To conclude, we now upper-bound $\mathbb{P}[F_m(h)]$. On $F_m(h)$, the following occurs:
\begin{eqnarray}
\label{E:eqnsumproba}
\sum_{z \in S(O(h),h)\cap \Cone_{O(h)}(I_\infty ,2^{m+1}e^{-h}) \cap \RST}\MBD_h^\infty(z)^p>2^{mp}e^{-ph},
\end{eqnarray}
thus, by Markov inequality,
\begin{eqnarray*}
\mathbb{P}\left[ F_m(h) \right] \le 2^{-mp}e^{ph}\mathbb{E}\left[ \sum_{z}\MBD_h^\infty(z)^p \right] \leq C_{\fluct} 2^d 2^{m(d-p)}
\end{eqnarray*}
where we used Proposition \ref{Prop:globalfluct} with $r_0=h$, $A=2^{m+1}$. The constant $C_{\fluct}>0$ appears in Proposition \ref{Prop:globalfluct} and depends only on $p$. Thus, for some $p>d$,
\begin{eqnarray*}
\mathbb{P}[E(2^n,h)] \ge 1-2^d C_{\fluct} \frac{2^{n(d-p)}}{1-2^{d-p}}.
\end{eqnarray*}The right hand side converges to 1 when $n\to \infty$, and 
this proves the first part of Proposition \ref{Prop:cylinder}. The second part be deduced from the first part by applying the dilation $(x,y) \mapsto (e^ax,e^ay)$, which is an isometry of $(H,ds^2)$.}
\end{proof}

\begin{proof}[Proof of Lemma \ref{Lem:geomlemma}]
Let $A,a \ge 0$ be fixed. For $h\geq 0$, let $z=(x,y) \in \Vois(A,h)$. Considering the totally geodesic plane containing $I_\infty,z$ and $O(h)$ (represented by a half-plane in $H$), it is possible to suppose $d=1$ without loss of generality. \dave{We apply the distance and angle formulas in $(H,ds^2)$ (see e.g. \cite[Proposition 2.1]{HyperbolicDSF}). \\

Let $z=(x,y) \in \Vois(A,h)$. We recall that the angle $\widehat{I_\infty O(h) z}$ can be computed as a function of $x$, $y$ and $h$:
\begin{eqnarray}
\widehat{I_\infty O(h) z}=\arctan\left|\frac{2xe^h}{e^{2h}-x^2-y^2} \right|.\label{eq:angle_appendix}
\end{eqnarray}
This formula can be obtained by noticing that the Poincar\'e open-ball model is conform and by using the fact that the application $\phi:H \to I$, sending the half-space model to the open-ball model, and defined as:
\begin{eqnarray*}
(x,y) \mapsto \frac{1}{x^2+(y+1)^2}\left( x^2+y^2-1,-2x \right)
\end{eqnarray*}
is an isometry sending $(0,1)$ on $(0,0)$. Then, the angle $\widehat{zO(h)I_\infty}=\widehat{\phi'(z)0\phi'(I_\infty})$, where the second angle is taken in the disc model. Since $\phi'(I_\infty)=(-1,0)$, if $y<e^h$ then $\widehat{z0(h)I_\infty}<\frac{\pi}{2}$ and we can establish \eqref{eq:angle_appendix}.\\
On the one hand, $\widehat{z O(h) I_\infty} \le Ae^{-h}$, so, taking $h$ large enough such that $Ae^{-h}<\pi/2$, we have
\begin{eqnarray*}
\arctan\left|\frac{2xe^h}{e^{2h}-x^2-y^2} \right| \le Ae^{-h}.
\end{eqnarray*}Thus, for $h$ large enough,
\begin{eqnarray*}
|x|e^{-h} \le \arctan|2xe^{-h}| \le \arctan\left|\frac{2xe^h}{e^{2h}-x^2-y^2} \right| \le Ae^{-h},
\end{eqnarray*}
so $|x| \le A$.\\
On the other hand, $d(O(h),z) \ge h$. 
Recall that for $z_1=(x_1,y_1)$ and $z_2=(x_2,y_2) \in H$, 
\begin{eqnarray}
\label{E:distanceformula}
d(z_1,z_2)=2\tanh^{-1}\left(\sqrt{\frac{\kappa^2+(v-1)^2}{\kappa^2+(v+1)^2}}\right)=2\tanh^{-1}\left(\sqrt{1-\frac{4v}{\kappa^2+(v+1)^2}}\right),
\end{eqnarray}where $\kappa=\|x_1-x_2\|/y_1$ and $v=y_2/y_1$. 
Applying this formula for $z=(x,y)$ and $O(h)=(0,e^h)$,
\begin{eqnarray*}
&&2\tanh^{-1}\left(\sqrt{1-\frac{4ye^h}{A^2+(y+e^h)^2}}\right) \\
&&\overset{|x| \le A}{\ge} 2\tanh^{-1}\left(\sqrt{1-\frac{4ye^{-h}}{(xe^{-h})^2+(ye^{-h}+1)^2}}\right)=d(O(h),z) \ge h.
\end{eqnarray*}
This implies that 
\[\sqrt{1-\frac{4ye^{-h}}{(xe^{-h})^2+(ye^{-h}+1)^2}}\geq \tanh\big(\frac{h}{2}\big)=\frac{e^{h/2}-e^{-h/2}}{e^{h/2}+e^{-h/2}},\]from which we deduce that:
\begin{eqnarray*}
&e^{h} &\le \frac{1+\sqrt{1-\frac{4ye^h}{A^2+(y+e^h)^2}}}{1-\sqrt{1-\frac{4ye^h}{A^2+(y+e^h)^2}}}
=\frac{e^h}{y}+o(e^h) \text{ when } h \to \infty,
\end{eqnarray*}
for $h$ large enough this implies $y \le 3/2$. The two other inclusions are shown by similar computations.}
\end{proof}

\section{Proof of Proposition \ref{Prop:controlannulus}}
\label{S:proofcontrolannulus}

We use a \tvc{block} control argument similar to \cite[Section 4.3, Proof of Prop. 4.6]{HyperbolicDSF}. Let $\delta>0$ small and $A>0$ large that will be chosen later. For $r_0>0$ and $z \in S(r_0)$, let us define
\begin{eqnarray}
\Psi_1(r_0,z)&:=&\Cone(z,3Ae^{-r_0}) \cap (B(r_0+\delta) \backslash B(r_0)), \nonumber\\
\Psi_2(r_0,z)&:=&\Cone(z,Ae^{-r_0}) \cap (B(r_0) \backslash B(r_0-1)).
\end{eqnarray}
A point $z \in S(r_0)$ is said to be \emph{good} if the following event $G(r_0,z)$ occurs (see Figure \ref{Fig:goodpoint}):
\begin{eqnarray}
G(r_0,z):=\left\{\mathcal{N} \cap \Psi_1(r_0,z)=\emptyset \text{ and } \mathcal{N}\cap \Psi_2(r_0,z) \neq \emptyset\right\}.
\end{eqnarray}

\begin{figure}
    \centering
    \begin{tikzpicture}[scale=4]
        \draw (0,0) node[below] {$0$};
        \draw (0,0) -- (5:1.6);
        \draw (0,0) -- (25:1.6);
        \draw (0,0) -- (45:1.6);
        \draw (0,0) -- (65:1.6);
        \draw (5:1.6) arc (5:65:1.6);
        \draw[blue, line width=2] (25:1.5) arc (25:45:1.5);
        \draw (5:1.5) arc (5:65:1.5);
        \draw (5:1) arc (5:65:1);
        \fill (37:1.3) circle (0.5pt);
        \draw (37:1.3) node[below] {$\in \mathcal{N}$};
        \draw (15:1.55) node {$\emptyset$};
        \draw (35:1.55) node {$\emptyset$};
        \draw (55:1.55) node {$\emptyset$};
        \fill (35:1.5) circle (0.5pt);
        \draw (35:1.5) node[left] {$z$};
        \draw (5:1) node[below] {$S(r_0-1)$};
        \draw (5:1.5) node[below] {$S(r_0)$};
        \draw (65:1.6) node[above] {$S(r_0+\delta)$};
        \draw[dashed] (0,0) -- (35:0.7);
        \draw (35:0.5) arc (35:25:0.5);
        \draw (35:0.5) node[right] {$Ae^{-r_0}$};
        \fill[pattern=dots] (65:1.6) -- (65:1.5) arc(65:5:1.5) -- (5:1.6) arc(5:65:1.6);
        \fill[pattern=crosshatch dots] (45:1.5) -- (45:1) arc(45:25:1) -- (25:1.5) arc(25:45:1.5);
        \filldraw[pattern=dots] (1.7,1)--(2,1)--(2,1.2)--(1.7,1.2) -- (1.7,1);
        \filldraw[pattern=crosshatch dots] (1.7,0.7)--(2,0.7)--(2,0.9)--(1.7,0.9) -- (1.7,0.7);
        \draw (2,1.1) node[right] {$\Psi^1(r_0,z)$};
        \draw (2,0.8) node[right] {$\Psi^2(r_0,z)$};
    \end{tikzpicture}
    \caption{The point $z$ is a \emph{good point}; the fluctuations of trajectories crossing $B_{S(r_0)}(z,Ae^{-r_0})$ (in \textcolor{blue}{blue}) are well controlled.}
    \label{Fig:goodpoint}
\end{figure}
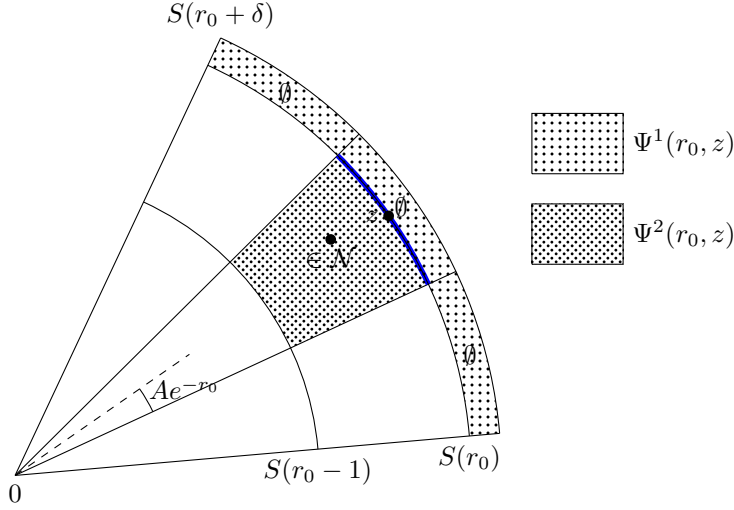

Let us define the random subsets $\hat{\chi}(r_0) \subset S(r_0)$ and $\chi(r_0) \subset \mathbb{H}^{d+1} \backslash \{0\}$ as
 \tvc{
\begin{eqnarray}
\hat{\chi}(r_0):=\big\{z \in S(r_0),\, G(r_0,z) \text{ occurs}\big\} , \quad \chi(r_0):=\bigcup_{z=(r_0;u) \in \hat{\chi}(r_0)} ]0,(r_0+\delta;u)) \subset \mathbb{H}^{d+1} \backslash \{0\}.
\end{eqnarray}}
The region $\chi(r_0)$ is the \emph{controlled region}, where the cumulative forward deviations in the annulus $B(r_0+\delta) \backslash B(r_0)$ will be upper-bounded. This control of fluctuations is given by the following lemma proved in Section \ref{S:proofoflemmas}, and from which we will obtain the upper bound of $\MBD_r^{r+\delta}(z)$ appearing in Proposition \ref{Prop:controlannulus}:
\tvc{\begin{lemma}
\label{Lemma:geometriccontrol}
There exists some deterministic constant $C_{\geom}>0$ such that for any $r_0 >0$,
\[\sup_{s\in [r_0,r_0+\delta]}\sup_{z\in \mathcal{L}_s \cap \chi(r_0)}\CFD_{r_0}^s(z) \le C_{\geom}e^{-r_0}.\]
\end{lemma}
}
For any $r_0 >0$, we cover the sphere $S(r_0)$ with balls of angular radius $e^{-r_0}$ such that the number of balls that overlap at a given point never exceeds some constant $K$, and such that the number of balls intersecting $B_{S(r_0)}(z,Ae^{-r_0})$ (recall \eqref{E:defBSR}) is upper-bounded by $C_{\ball}A^d$. To proceed, we use the following lemma, proved in Section \ref{S:proofoflemmas}:
\begin{lemma}
\label{Lemma:covering}
There exists $K=K(d,p) \in \mathbb{N}^*$ such that, for any $r_0 >0$, there exist a non-negative integer $N(r_0) \ge 0$ and a family of points $z_1,...,z_{N(r_0)} \in S(r_0)$ such that:
\begin{itemize}
    \item $\bigcup_{1 \le i \le N(r_0)} B_{S(r_0)}(z_i,e^{-r_0})=S(r_0)$,
    \item $\forall z \in S(r_0)$, $\#\{1 \le i \le N(r_0),~z \in B_{S(r_0)}(z_i,e^{-r_0})\} \le K$.
\end{itemize}
Moreover, there exists $C_{\ball}=C_{\ball}(K,d)>0$, such that for all $r_0 >0,$  and $A\geq 1$, 
\begin{eqnarray*}
\sup_{z \in S(r_0)}\#\{1 \le i \le N(r_0), B_{S(r_0)}(z_i,e^{-r_0}) \cap B_{S(r_0)}(z,Ae^{-r_0}) \neq \emptyset\} \le C_{\ball}A^d.
\end{eqnarray*}
\end{lemma}

For $1 \le i \le N(r_0)$, $z_i$ is said to be \emph{inhibited} if the ball $B_{S(r_0)}(z_i,e^{-r_0})$ intersects $S(r_0) \backslash \hat\chi \dave{(r_0)}$, and the corresponding event is denoted by $\Activ(i)$. Let $\Psi(r_0) \subset \llbracket 1,N(r_0) \rrbracket$ be the union of all inhibited spherical caps:
\begin{eqnarray*}
\Psi(r_0):=\bigcup_{\substack{1 \le i \le N(r_0),\\B_{S(r_0)}(z_i,e^{-r_0}) \cap \left( S(r_0) \backslash \hat\chi \dave{(r_0)}\right) \neq\emptyset}}B_{S(r_0)}(z_i,e^{-r_0}).
\end{eqnarray*}
The region $\Psi\dave{(r_0)}$ is the \emph{augmented uncontrolled region}, that contains $\left(\mathbb{H}^{d+1} \backslash \{0\}\right) \backslash \chi\dave{(r_0)}$. For $z \in S(r_0)$, let $\hat{\Cl}(z)$ be the cluster of $z$ in $\Psi(r_0)$ and let us also define (recall that $]0,z)$ is the semi-geodesic starting at $0$ and containing $z$, without $0$):
\begin{eqnarray*}
\Cl(z)=\bigcup_{\substack{z \in \hat{\Cl}(z), \\ z=(r_0 ; u)}} ]0,\dave{(r_0+\delta ; u)}) \subset \mathbb{H}^{d+1} \backslash \{0\}.
\end{eqnarray*}
We define the \emph{angular radius} of $\Cl(z)$ as 
\begin{equation}\label{eq:def_Rad(z)}
\Rad(z):=\sup_{z' \in \Cl(z)} \widehat{z'0z}.
\end{equation}

The next lemma, proved  in Section \ref{S:proofoflemmas}, asserts that the connected components of the augmented uncontrolled region $\Psi$ are small (the radius admits exponential tail decay):
\begin{lemma}
\label{Lemma:expdeca}
There exist $\delta>0$ small enough, $A>0$ large enough and some constant $c_{\decay}>0$ such that, for any $B>0$ large enough and $r_0 >0$,
\begin{eqnarray*}
\sup_{z\in S(r_0)}\mathbb{P}\left[ e^{r_0}\Rad(z)>B\right] \le e^{-c_{\decay}B}.
\end{eqnarray*}
\end{lemma}

In addition, we need a control of the number of points in a given region of a sphere $S(r_0)$, which is given by the next lemma, proved in Section \ref{S:proofoflemmas}:
\begin{lemma}
\label{Lemma:numberofpoints}
For any $p \ge 1$, there exists a constant $C=C(d,p)>0$ such that, for any $r_0 > 0$,
\begin{eqnarray*}
\sup_{z \in S(r_0)}\mathbb{E}\left[ \#\left(\mathcal{L}_{r_0} \cap B_{S(r_0)}(z,e^{-r_0})\right)^p \right] \le C.
\end{eqnarray*}
\end{lemma}
Let us choose $A,\delta>0$ as in Lemma \ref{Lemma:expdeca} and $C_{\geom}$ as in Lemma \ref{Lemma:geometriccontrol}.
\dave{The proof of Proposition \ref{Prop:controlannulus} is now divided into steps. First, we will prove in Section \ref{sec:step1} an estimate on the maximal backward deviations in the control region, between the radii $r_0$ and $r_0+\delta$ using Lemma \ref{Lemma:geometriccontrol}. This upper bound involves the number of points $M(z)$ of a random region of the annulus $B(r_0+\delta) \backslash B(r_0)$ (see \eqref{E:defineM} below). The moments of this random variable are controlled in Section \ref{sec:Step2}, using in particular the result of Lemma \ref{Lemma:expdeca}. In the computation, Lemma \ref{Lemma:numberofpoints} will be useful. This allows us to conclude.}

\subsection{Step 1: an almost-sure upper-bound of $\MBD_{r_0}^{r_0+\delta}(\cdot)$}\label{sec:step1}

For $z \in S(r_0)$, let us define
\begin{eqnarray}
\label{E:defineM}
M(z) := \# \Big( \mathcal{N} \cap \Cl(z) \cap \tvc{(B(r_0+\delta) \backslash B(r_0)) \Big) ~.}
\end{eqnarray}
The quantity $M(z)$ is the number of Poisson points inside a random part of the thin annulus $B(r_0+\delta) \backslash B(r_0)$, whose diameter admits exponential tail decay (Lemma \ref{Lemma:expdeca}).

This step is devoted to the proof of the following upper-bound: 
\begin{lemma}
Almost surely, for any $z \in S(r_0)$ and $p\geq 1$,
\begin{eqnarray}
\label{E:majorMBD}
\MBD_{r_0}^{r_0+\delta}(z)^p \le 2^{p-1}\big(2^p \Rad(z)^p(M(z)+1)^p+C_{\geom}^pe^{-pr_0}\big).
\end{eqnarray}
\end{lemma}

\begin{proof}
Recall that $\MBD_{r_0}^{r_0+\delta}(z)$ takes into account backward paths from $z$ that end (in the backward direction) before level $r_0+\delta$ and backward paths reaching level $r_0+\delta$. For $z \in S(r_0)$, we define $\Stop(z)$ as the set of ending points of backward paths from $z$ stopping before level $r_0+\delta$:\tvc{
\begin{eqnarray}
\Stop(z):=\{z'=(r';u') \in \mathcal{N} \cap \mathcal{D}_{r_0}^{r'}(z),~r_0 \le r'<r_0+\delta,~A^{(-1)}(z')=\emptyset\}.\label{eq:def_Stop}
\end{eqnarray}}
For any $z' \in \mathcal{D}_{r_0}^{r_0+\delta}(z) \cup \Stop(z)$, one the two following cases occur. Either the branch from $z'$ to $z$ stays inside $\Cl(z)$, or it crosses $\chi(r_0)$. Let us define $\mathcal{C}$ (resp. $\mathcal{C}'$) as the set of couples $(z,z')$ such that the branch from $z'$ to $z$ crosses $\chi(r_0)$ (resp. does not cross $\chi(r_0)$):
\begin{eqnarray*}
\mathcal{C}:=\{(z,z'),~z \in S(r_0),~z' =(r',u') \in \mathcal{D}_{r_0}^{r_0+\delta}(z) \cup \Stop(z),~\exists s \in [r_0,r'],~\mathcal{A}_s^{r'}(z') \in \chi(r_0)\}, \\ 
\mathcal{C}':=\{(z,z'),~z \in S(r_0),~z'=(r',u')  \in \mathcal{D}_{r_0}^{r_0+\delta}(z) \cup \Stop(z),~\forall s \in [r_0,r'],~\mathcal{A}_s^{r'}(z') \notin \chi(r_0)\}.
\end{eqnarray*}
Moreover, for any $(z,z') \in \mathcal{C}$, we define $\hit(z,z')$ as the highest level where the branch from $z'=(r',u')$ to $z$ hits $\chi(r_0)$:
\begin{eqnarray*}
\hit(z,z'):=\sup \{s \in [r_0,r_0+\delta],~\mathcal{A}_s^{r'}(z') \in \chi(r_0)\}.
\end{eqnarray*}
\tvc{By definition, for $(z,z') \in \mathcal{C}$, the set on which the supremum is non empty and bounded by $r_0+\delta$, so the supremum is well defined and smaller than $r_0+\delta$.}\\

\tvc{Let $1 \le i \le N(r_0)$. Let $z \in B_{S(r_0)}(z_i,e^{-r_0})$. Using the above notations and Definition \ref{def:defmbd}, we recall that:
\[\MBD_{r_0}^{r_0+\delta}(z)=\max_{z'=(r',u')\in \mathcal{C} \cup \mathcal{C}' } \CFD_{r_0}^{r'}(z').\]
We will now establish that for every $z'=(r';u') \in \mathcal{D}_{r_0}^{r_0+\delta}(z) \cup \Stop(z)$, $\CFD_{r_0}^{r'}(z')$ is upper bounded by the right hand side of \eqref{E:majorMBD}. Once this is proved, \eqref{E:majorMBD} follows as the upper bound does not depend on $z'$.}\\

Let us first consider the case where $(z,z') \in \mathcal{C}'$ (the branch between $z$ and $z'$ does not cross $\chi(r_0)$). Then this branch stays inside $\Cl(z)$. Thus it crosses at most $M(z)$ points of $\mathcal{N}$, then, for $r_0\leq r'\leq r_0+\delta$,
\begin{eqnarray*}
\CFD_{r_0}^{r'}(z') \le 2\Rad(z)(M(z)+1).
\end{eqnarray*}
In the second case where $(z,z') \in \mathcal{C}$, let $z''=(r'';u'')$ be the ancestor of $z'$ at the level $r''=\hit(z,z')$. Let $p \ge 1$. Then, 
\begin{eqnarray}
\CFD_{r_0}^{r'}(z')^p \le \left(\CFD_{r_0}^{r''}(z'')+\CFD_{r''}^{r_0+\delta}(z')\right)^p \le 2^{p-1}\left(\CFD_{r_0}^{r''}(z'')^p+\CFD_{r''}^{r_0+\delta}(z')^p\right).\label{eq:etape11}
\end{eqnarray}
By the same argument as in the previous case, $\CFD_{r_0}^{r''}(z'') \le 2\Rad(z)(M(z)+1)$,
and, by Lemma \ref{Lemma:geometriccontrol}, $\CFD_{r''}^{r_0+\delta}(z') \le C_{\geom}e^{-r_0}$, since, by definition of $z''$, the part of trajectory between $z''$ and $z$ is included in $\Cl(z)$. Thus,
\begin{eqnarray}
\label{E:majorCFD}
\CFD_{r_0}^{r'}(z')^p \le 2^{p-1}\left(2^p \Rad(z)^p(M(z)+1)^p+C_{\geom}^pe^{-pr_0}\right).
\end{eqnarray}
The upper-bounds \eqref{eq:etape11} and (\ref{E:majorCFD}) hold whatever $(z,z')$ belongs to $\mathcal{C}$ or $\mathcal{C}'$, and does not depend on $z'$. It follows that (\ref{E:majorMBD}) holds for any $z \in S(r_0)$.\end{proof}

\subsection{Step 2: a control of the tail decay of $M(z)$}\label{sec:Step2}

Recall that, for $z =(r;u)\in S(r_0)$, $M(z)$ is defined in (\ref{E:defineM}). In this step, it is shown that:

\begin{lemma}For any $z \in S(r_0)$, and for $p\geq 1$,
\begin{eqnarray}
\label{E:boundCseconde}
\mathbb{E}\left[M(z)^{4p}\right] \le C
\end{eqnarray}
for some $C=C(p)>0$.\end{lemma}

\begin{proof}
Let $z \in S(r_0)$. For given $R \ge 0$, let us define
\begin{eqnarray*}
\Reg(R):=\Cone(z_i,Re^{-r_0}) \cap \big(\dave{B}(r_0+\delta)\backslash \dave{B}(r_0)\big).
\end{eqnarray*}
For any $R,m \ge 0$,
\begin{eqnarray*}
\{M(z)>m\} \subset \{\Rad(z)>Re^{-r_0}\} \cup \{\#\left(\mathcal{N} \cap \Reg(R) \right)>m\}
\end{eqnarray*}
thus
\begin{eqnarray}
\label{E:boundprobasum}
\mathbb{P}\left[M(z)>m\right] \le \mathbb{P}\left[ \Rad(z)>Re^{-r_0} \right]+\mathbb{P}\left[\#\left(\mathcal{N} \cap \Reg(R)\right)>m\right].
\end{eqnarray}
By Lemma \ref{Lemma:expdeca}, $\mathbb{P}\left[ \Rad(z)>Re^{-r_0} \right] \le e^{-c_{\decay}R}$. The random variable $\#\left(\mathcal{N} \cap \Reg(R))\right)$ is distributed according to the Poisson law with parameter $\lambda \Vol(\Reg(R))$. Recall that $\nu$ the $d$-dimensional volume measure on $\mathbb{S}^d$, and that $u$ is the direction of $z$. Then:
\begin{eqnarray}
\lambda \dave{\Vol}(\Reg(R)) = \lambda \nu(\{u'\ : \ \widehat{u0u'} < Re^{-r_0}\})\Vol(\dave{B}(r_0+\delta)\backslash \dave{B}(r_0)) \le CR^d.
\end{eqnarray}
for some constant $C>0$ independent of $r_0,R$, since $\Vol(\dave{B}(r_0+\delta)) =O\left(e^{dr_0}\right)$ by (\ref{E:volball}). Thus $\#(\mathcal{N} \cap \Reg(R)) \preceq_{st} \mathcal{P}(CR^d)$. By Chernoff bound for the Poisson distribution \cite{boundPoisson},
\begin{eqnarray}
\mathbb{P}\left[ \#(\mathcal{N} \cap \Reg(R)) \ge m\right] \le \frac{e^{-CR^d}(CR^d e^1)^m}{m^m}
\end{eqnarray}
for any $m \ge CR^d$. Let us chose $R=(m/(2eC))^{1/d}$ (thus $m=2eCR^d$). It leads to:
\begin{eqnarray}
\label{E:finalboundReg}
\mathbb{P}\left[ \#(\mathcal{N} \cap \Reg(R)) \ge m\right] \le \left(\frac{e^{-1/(2e)}}{2}\right)^m \le \left(\frac{1}{2} \right)^m.
\end{eqnarray}
Finally, we combine (\ref{E:boundprobasum}), Lemma \ref{Lemma:expdeca} and (\ref{E:finalboundReg}) to obtain:
\begin{eqnarray}
\label{E:queuenopcl}
\mathbb{P}\left[\dave{M(z)} > m\right] & \le  & e^{-c_{\decay}R}+\left(\frac{1}{2} \right)^m \nonumber\\
&= & \exp\left(-c_{\decay}(m/2eC)^{1/d}\right)+\left(\frac{1}{2}\right)^m \le \exp(-cm^{1/d})
\end{eqnarray}
for some $C>0$. Therefore:
\begin{eqnarray*}
\mathbb{E}\left[ (\dave{M(z)})^{4p} \right] 
&= & \int_0^\infty \mathbb{P}\left[\dave{M(z)}>m^{1/(4p)}\right]~dm \nonumber\\
&\overset{(\ref{E:queuenopcl})}{\le} & \int_0^\infty \exp(-cm^{1/(4dp)})~dm<\infty,
\end{eqnarray*}
which proves (\ref{E:boundCseconde}).\end{proof}

\subsection{Step 3: conclusion}

\dave{By (\ref{E:majorMBD}), for any $1 \le i \le N(r_0)$,
\begin{multline}
\sum_{z \in B_{S(r_0)}(z_i,e^{-r_0}) \cap \RST} \MBD_{r_0}^{r_0+\delta}(z)^p 
\\
\begin{aligned}
 \le & 2^{p-1} \mathbb{E}\Big[\sum_{z\in B_{S(r_0)}(z_i,e^{-r_0}) \cap \RST}\left(2^p \Rad(z)^p(M(z)+1)^p+C_{\geom}^pe^{-pr_0}\right) \Big]\\
 \leq & I+ II,
\end{aligned}
\end{multline}where 
\begin{align*}
I= & 2^{2p-1}\mathbb{E}\Big[\sum_{z\in B_{S(r_0)}(z_i,e^{-r_0}) \cap \RST}  \Rad(z)^p(M(z)+1)^p \Big]\\
II=& 2^{p-1} C_{\geom}^pe^{-pr_0} \mathbb{E}\big[\# \{\RST\cap B_{S(r_0)}(z_i,e^{-r_0})\}\big]
\end{align*}
By Lemma \ref{Lemma:numberofpoints}, the expectation in $II$ is upperbounded by a constant $C>0$. For the term $I$, notice that by construction of $\Psi(r_0)$, $\Cl(z)=\Cl(z_i)$, $\Rad(z)\leq \Rad(z_i)+e^{-r_0}$ and $M(z)=M(z_i)$. Thus, we have:
\begin{align}
I\leq & 2^{2p-1} \mathbb{E}\Big[ \sum_{z\in B_{S(r_0)}(z_i,e^{-r_0}) \cap \RST}  (e^{-r_0}+\Rad(z_i))^p(M(z_i)+1)^p\Big]\nonumber\\
= & 2^{2p-1} \mathbb{E}\Big[ \#\{\RST \cap B_{S(r_0)}(z_i,e^{-r_0}) \}  (e^{-r_0}+\Rad(z_i))^p(M(z_i)+1)^p\Big]\nonumber\\
\leq & 2^{2p-1} \mathbb{E}\Big[ \#\{\RST \cap B_{S(r_0)}(z_i,e^{-r_0}) \}^2\Big]^{1/2} \mathbb{E}\Big[ (e^{-r_0}+\Rad(z_i))^{2p}(M(z_i)+1)^{2p}\Big]^{1/2}\nonumber\\
\leq & 2^{2p-1}\mathbb{E}\Big[ \#\{\RST \cap B_{S(r_0)}(z_i,e^{-r_0}) \}^2\Big]^{1/2} \mathbb{E}\Big[ (e^{-r_0}+\Rad(z_i))^{4p}\Big]^{1/4} \mathbb{E}\Big[(M(z_i)+1)^{4p}\Big]^{1/4}, \label{E:majorMBDproduct}
\end{align}by using Cauchy-Schwarz' inequality. By Lemma \ref{Lemma:numberofpoints} applied to $p=2$,
\begin{eqnarray}
\label{E:majornbp}
\mathbb{E}\left[ \#\{\RST \cap B_{S(r_0)}(z_i,e^{-r_0})\}^2 \right] \le C,
\end{eqnarray}
for $C$ independent of $r_0$, $z_i$. By Lemma \ref{Lemma:expdeca},
\begin{eqnarray}
\label{E:majorRad}
&\mathbb{E}\left[ \Rad(z_i)^{4p} \right]&\leq e^{-4pr_0}\int_0^\infty \mathbb{P}\left[e^{4pr_0}\Rad(z_i)^{4p}>B \right]~dB \nonumber\\
&&=e^{-4pr_0}\int_0^\infty \mathbb{P}\left[ e^{r_0}\Rad(z_i)>B^\frac{1}{4p} \right]~dB \nonumber\\
&&\le e^{-4pr_0}\int_0^\infty e^{-c_{\decay}B^\frac{1}{4p}}~dB<\infty.
\end{eqnarray}
Then, by combining (\ref{E:majorMBDproduct}), (\ref{E:majornbp}), (\ref{E:majorRad}) and (\ref{E:boundCseconde}), $I\leq C(p)e^{-pr_0}$ and 
\begin{eqnarray}
\label{E:majorunebouleelem}
\mathbb{E}\left[ \sum_{z \in B_{S(r_0)}(z_i,e^{-r_0}) \cap \RST} \MBD_{r_0}^{r_0+\delta}(z)^p\right] \le Ce^{-pr_0},
\end{eqnarray}
for some $C=C(p)>0$.}

The final step is to sum over all $i$ such that $B_{S(r_0)}(z_i,e^{-r_0})$ intersects $B_{S(r_0)}(u,\theta)$ for any given $u \in \mathbb{S}^d$ and $\theta>0$. Let $\theta>0$. By Lemma \ref{Lemma:covering} it can be assumed that $B_{S(r_0)}(z_i,e^{-r_0})$ intersects $B_{S(r_0)}(u,\theta)$ for at most $C_{\ball}e^{dr_0}\theta^d$ values of $i \in \llbracket 1,N(r_0) \rrbracket$. Therefore
\begin{multline}
\mathbb{E}\left[ \sum_{z \in B_{S(r_0)}(u,\theta) \cap \RST} \MBD_{r_0}^{r_0+\delta}(z)^p\right]\\
\begin{aligned}
= & \sum_{\substack{1 \le i \le N(r_0),\\B_{S(r_0)}(z_i,e^{-r_0}) \\ \cap B_{S(r_0)}(u,\theta) \neq \emptyset}} \mathbb{E}\left[ \sum_{z \in B_{S(r_0)}(z_i,e^{-r_0}) \cap \RST} \MBD_{r_0}^{r_0+\delta}(z)^p\right] \\
\overset{(\ref{E:majorunebouleelem})}{\le} & Ce^{-pr_0}\#\{1 \le i \le N(r_0),B_{S(r_0)}(z_i,e^{-r_0})\cap B_{S(r_0)}(u,\theta) \neq \emptyset\} \\
\le & Ce^{(d-p)r_0}\theta^d,
\end{aligned}
\end{multline}
which achieves the proof of Proposition \ref{Prop:controlannulus}.

\subsection{Proof of Lemmas \ref{Lemma:geometriccontrol}, \ref{Lemma:covering}, \ref{Lemma:expdeca}, \ref{Lemma:numberofpoints}}
\label{S:proofoflemmas}

\tvc{
\begin{proof}[Proof of Lemma \ref{Lemma:geometriccontrol}]
Let $r_0 \le s \le r_0+\delta$ and $z =(s;u) \in \chi(r_0) \cap \mathcal{L}(s)$. Recall that the edge of the RST containing $z$ is $[z_\uparrow,z_\downarrow]$, with $z_{\uparrow}=A(z)$. 

We will first prove that the ancestor $z_{\uparrow}$ belongs to $B(r_0)$. 
Let us introduce $z_1=(r_0 ; u)$ and $z_2=(r_0+\delta,u)$, and let $R=\inf\{\rho>0\ : \ B(z_2,\rho)\supset \Psi_2(r_0,z_1)\}$ be the smallest radius such that $\Psi_2(r_0,z_1)$ is included into $B(z_2,R)$.\\
Let us denote by $V:= B(z_2,R)\cap (B(r_0+\delta)\setminus B(r_0))$. For sufficiently large $A$ in the definition of $\Psi_1(r_0,z_1)$ and $\Psi_2(r_0,z_1)$ ($A\geq 2$), this region is included into $\Psi_1(r_0,z_1)$.\\
Since $z\in \chi(r_0)$, $G(r_0,z_1)$ is true: there is no Poisson point in $\Psi_1(r_0,z_1)$ and at least one point, say $z'$ in $\Psi_2(r_0,z_1)$. Hence, the ancestor of $z$ is necessarily in $B(z_2,R)$ and since it can not be in $V$, it is in $B(r_0)$.
Therefore, the radius $r_\uparrow$ of $z_\uparrow$ is smaller than $r_0$, and $z_3:=\mathcal{A}_{r_0}^s(z)$ of radius $r_0$ is well-defined. \\

The points $z_\uparrow$, $z_3$, $z$ and $z_\downarrow$ are all on the same edge of the RST. Our purpose is now to prove that 
\begin{equation}\label{etape2}
\CFD_{r_0}^s(z)=\widehat{z0z_3}\leq A e^{-r_0}.
\end{equation} 
We proceed by contradiction. Assume that $\widehat{z0z_3}>A e^{-r_0}$. On the one hand, since $z$ is `between' $z_3$ and $z_\downarrow$
\begin{align*}
\widehat{z_\downarrow 0 z_3}= & \widehat{z_\downarrow 0z}+\widehat{z0z_3} > \widehat{z_\downarrow 0z}+A e^{-r_0},
\end{align*}by assumption. On the other hand, by construction of the point $z'=(r';u')\in \Psi_2(r_0,z_1)$,
\begin{align*}
\widehat{z_\downarrow 0 z'}\leq & \widehat{z_\downarrow 0z}+\widehat{z_\downarrow 0 z'}\leq \widehat{z_\downarrow 0 z}+A e^{-r_0}.
\end{align*}Hence, 
$\widehat{z_\downarrow 0 z_3}>\widehat{z_\downarrow 0 z'}$ and
\[\big(1-\cos(\widehat{z_\downarrow 0 z'})\big) \sinh(r_\downarrow) \sinh(r')<\big(1-\cos(\widehat{z_\downarrow 0 z_3})\big) \sinh(r_\downarrow) \sinh(r_0),\]
and by the hyperbolic law of cosines, 
$\cosh\big(d(z_\downarrow,z')\big)<\cosh\big(d(z_\downarrow,z_3)\big)$ which implies that $d(z_\downarrow,z') <d(z_\downarrow,z_3)$. Since the distance to $z_\downarrow$ is increasing along the edge $[z_\downarrow,z_\uparrow]$ by the construction \eqref{E:defofstar}, we have that $d(z_\downarrow,z_3)\leq d(z_\downarrow,z_\uparrow)$ and thus,
\[d(z_\downarrow,z') <d(z_\downarrow,z_3)\leq d(z_\downarrow,z_\uparrow).\]
This is contradictory since $z'\in \mathcal{N}$ is a point of lower radius than $z_\uparrow$ which is strictly closer to $z_\downarrow$ than $z_\uparrow$ is: then $z'$ should be the ancestor of $z_\downarrow$ and not $z_\uparrow$! As a consequence, we deduce that \eqref{etape2} holds. This concludes the proof of Lemma \ref{Lemma:geometriccontrol}.\end{proof}}

\bigbreak

\begin{proof}[Proof of Lemma \ref{Lemma:covering}]
Proving the first part of Lemma \ref{Lemma:covering} is equivalent to show that there exists some $K \in \mathbb{N}$ such that, for any $\varepsilon>0$, the Euclidean unit sphere $\mathbb{S}^{d}$ can be covered by balls of radius $\varepsilon$ such that the number of balls overlapping some given point $x \in \mathbb{S}^d$ is bounded by $K$, which is a standard fact.

We move on to show the second part, i.e. the existence of $C_{\ball}>0$ such that, for any $r_0 >0$, $z \in S(r_0)$ and $A \ge 1$, the number of balls intersecting $B_{S(r_0)}(z,Ae^{-r_0})$ is upper-bounded by $C_{\ball}A^d$. Let $u_0 \in \mathbb{S}^d$ be the direction of $z$ and let $A \ge 1$. For $i \in \llbracket 1,N(r_0) \rrbracket$, the ball $B_{S(r_0)}(z_i,e^{-r_0})$ intersects $B_{S(r_0)}(z,Ae^{-r_0})$ if and only if $\widehat{z_i 0 z} \le (A+1)e^{-r_0}$. Thus
\begin{eqnarray*}
\bigcup_{\substack{1 \le i \le N(r_0),\\B_{S(r_0)}(z_i,e^{-r_0})\cap \\B_{S(r_0)}(z,Ae^{-r_0}) \neq \emptyset}} B_{S(r_0)}(z_i,e^{-r_0}) \subset B_{S(r_0)}(z,(A+2)e^{-r_0}).
\end{eqnarray*}
Recall that $\nu$ denotes the $d$-dimensional volume measure on $\mathbb{S}^d$. There exists $C>0$ such that, for any \dave{$r_0 >0$},
\begin{eqnarray*}
\nu\{u,~\widehat{u u_0} \le e^{-r_0}\} \ge Ce^{-r_0d},
\end{eqnarray*}\dave{where $\widehat{u u_0}$ is the angle made by the directions $u$ and $u_0$.}
Moreover,
\begin{eqnarray*}
\nu\{u,~\widehat{u u_0} \le (A+2)e^{-r_0}\} \le \left((A+2)e^{-r_0} \right)^d,
\end{eqnarray*}
thus the number of balls intersecting $B_{S(r_0)}(z,Ae^{-r_0})$ is upper-bounded by:
\begin{eqnarray*}
K\frac{\nu\{u,~\widehat{u u_0} \le (A+2)e^{-r_0}\}}{\nu\{u,~\widehat{u u_0} \le e^{-r_0}\}} \le \frac{K\left((A+2)e^{-r_0} \right)^d}{ce^{-r_0d}} \overset{A \ge 1}{\le} \frac{3^dK}{c} A^d,
\end{eqnarray*}
the conclusion follows.
\end{proof}

\bigbreak

\begin{proof}[Proof of Lemma \ref{Lemma:expdeca}]
Recall that, for $1 \le i \le N(r_0)$, $B_{S(r_0)}(z_i,e^{-r_0})$ is said to be inhibited if $B_{S(r_0)}(z_i,e^{-r_0}) \cap (S(r_0) \backslash \chi) \neq \emptyset$. Let us first estimate the probability that a given $z_i$ is inhibited.

Let  $1 \le i \le N(r_0)$. Let us consider the following events:
\begin{eqnarray*}
&&E(i):=\{\mathcal{N} \cap \Cone(z_i,(3A+1)e^{-r_0}) \cap (B(r_0+\delta) \backslash B(r_0))=\emptyset\}, \\
&&E'(i):=\{\mathcal{N} \cap \Cone(z_i,(A-1)e^{-r_0}) \cap (B(r_0) \backslash B((r_0-1)\wedge 0)) \neq \emptyset\}.
\end{eqnarray*}
We now show that $E(i) \cap E'(i) \subset \Activ(i)^c$. Let $z \in B_{S(r_0)}(z_i,e^{-r_0})$. By triangular inequality,
\begin{eqnarray*}
\Psi_1(r_0,z) \subset \Cone(z_i,(3A+1)e^{-r_0}) \cap (B(r_0+\delta) \backslash B(r_0)) &\text{ and }\\
\Cone(z_i,(A-1)e^{-r_0}) \cap (B(r_0+\delta) \backslash B(r_0)) \subset \Psi_2(r_0,z).
\end{eqnarray*}
Therefore, on the event $E(i) \cap E'(i)$, $z$ is good (i.e. $G(r_0,z)$ occurs). Thus, $\Activ(i)^c$ occurs, which shows that $E(i) \cap E'(i) \subset \Activ(i)^c$. It follows that $\mathbb{P}[\Activ(i)] \le \mathbb{P}[E(i)^c]+\mathbb{P}[E'(i)^c]$. Since
\begin{eqnarray*}
&&\Vol(\Cone(z_i,(3A+1)e^{-r_0}) \cap (B(r_0+\delta) \backslash B(r_0)))\\
&&=\int_{r_0}^{r_0+\delta} \nu\{u,~\widehat{u_i u}  \le (3A+1)e^{-r_0}\}\sinh(r)^d~dr \le C\delta A^d
\end{eqnarray*}
for some $C>0$ independent of $A,r_0,\delta$ and where $u_i$ is the direction of $z_i$. Thus
\begin{eqnarray*}
\mathbb{P}\left[ E(i)^c \right] \le 1-e^{-\lambda C\delta A^d}
\end{eqnarray*}
An analogous computation for $E'(i)$ leads to:
\begin{eqnarray*}
\mathbb{P}\left[ E'(i)^c \right] \le e^{-\lambda c A^d}
\end{eqnarray*}
for some $c>0$ independent of $A,r_0,\delta$.

\medbreak

We move on to show that $A$ and $\delta$ can be chosen such that
$\mathbb{P}[\Activ(i)] \le 2C_{\ball}^{-1}(6A+4)^{-d}/3$ for any $1 \le i \le N(r_0)$. Indeed, let us first chose $A$ such that $e^{-\lambda c A^d}C_{\ball}(6A+4)^d \le 1/3$. Then let us chose $\delta$ such that $1-e^{-\lambda C\delta A^d} \le C_{\ball}^{-1}(6A+4)^{-d}/3$. Hence
\begin{eqnarray*}
\tvc{\mathbb{P}[\Activ(i)]} \le \mathbb{P}\left[ E(i)^c \right]+\mathbb{P}\left[ E'(i)^c \right] \le 1-e^{-\lambda c\delta A^d}+e^{-\lambda c A^d} \le \frac{2}{3}C_{\ball}^{-1}(6A+4)^{-d}.
\end{eqnarray*}

We finally show that, for this choice of $A, \delta$, there exists $c_{\decay}>0$ such that for any $B$ large enough, $r_0 >0$ and $z \in S(r_0)$,
\begin{eqnarray*}
\mathbb{P}\left[ e^{r_0}\Rad(z)>B\right] \le e^{-c_{\decay}B}.
\end{eqnarray*}
Fix $r_0 >0$ and $z \in S(r_0)$. For given $k$, let us denote by $P(k)$ the set of sequences $z_{i_0}, \cdots, z_{i_k}$ among the $\{z_i,~1 \le i \le N(r_0) \}$ such that:
\begin{enumerate}
    \item $d(z_{i_0},z) \le e^{-r_0}$,
    \item $ \sup_{0\leq j\leq k-1} d(z_{i_j},z_{i_{j+1}}) \le (6A+4)e^{-r_0}$,
    \item $\sup_{0 \le j,j' \le k} d(z_{i_j},z_{i_{j'}}) \ge (6A+2)e^{-r_0}$.
\end{enumerate}
Let us also denote by $\hat{P}(k) \subset P(k)$ the set of sequences $z_{i_0}, \cdots, z_{i_k}$ verifying 1., 2., 3. and such that:

4. $z_{i_0},\cdots,z_{i_k}$ are inhibited.
\medbreak
It can be noticed that $\{\Rad(z) \ge k(6A+4)\} \subset \{\hat{P}(k) \neq \emptyset\}$ for any $k \in \mathbb{N}$, thus it is enough to upper-bound $\{\hat{P}(k) \neq \emptyset\}$ for any $k \in \mathbb{N}$.
Let $(z_{i_0}, \cdots, z_{i_k}) \in P(k)$. For $0 \le j \le k$, the event $\Activ(i_j)$ only depends on the Poisson process $\mathcal{N}$ inside $\Cone(z_{i_j},(3A+1)e^{-r_0})$, therefore, by 3. the events $\Activ(i_j)$ are mutually independent. Thus
\begin{eqnarray*}
\mathbb{P}[(z_{i_0},\cdots z_{i_k}) \in \hat{P}(k)]=\mathbb{P}[\Activ(i_0)]^{k+1} \le \left(\frac{2}{3}\right)^{k+1} C_{\ball}^{-(k+1)}(6A+4)^{-d(k+1)}.
\end{eqnarray*}
By Lemma \ref{Lemma:covering}, for any $1 \le i \le N(r_0)$, the number of balls intersecting $B_{S(r_0)}(z,\theta)$ is upper-bounded by $C_{\ball}e^{dr_0}\theta^d$. Thus
\begin{eqnarray*}
\#P(k) \le C_{\ball}^{k+1}(6A+4)^{d(k+1)}.
\end{eqnarray*}
It follows that
\begin{eqnarray*}
&&\mathbb{P}[\Rad(z) \ge k(6A+4)] \le \mathbb{P}[\hat{P}(k) \neq \emptyset] \le \mathbb{E}[\#\hat{P}(k)] \\
&&\le  C_{\ball}^{k+1}\left(\frac{2}{3}\right)^{k+1}(6A+4)^{d(k+1)}C_{\ball}^{-(k+1)}(6A+4)^{-d(k+1)} \le \left( \frac{2}{3} \right)^{k+1}.
\end{eqnarray*}
Lemma \ref{Lemma:expdeca} follows.
\end{proof}

\bigbreak

\begin{proof}[Proof of Lemma \ref{Lemma:numberofpoints}]
Let $r_0,M>0$ and $z \in S(r_0)$. Let $h \ge 0$ that will be fixed later. We divide the set \tvc{$L=\{z' \in \mathcal{N}\ : \ [z',A(z')]^* \cap B_{S(r_0)}(z,e^{-r_0}) \not=\emptyset \}$} into two subsets $L_{\le h}$ and $L_{>h}$ according to the length of $[z',A(z')]^*$:
\begin{eqnarray}
L_{\le h}:=\{z' \in L,~d(z',z) \le h\}, \quad L_{>h}:=\{z' \in L,~d(z',z)>h\}.
\end{eqnarray}
Thus $L=L_{\le h} \cup L_{>h}$, and
\begin{eqnarray}
\label{E:sumprobalh}
\mathbb{P}\left[\#L \ge M \right] \le \mathbb{P}\left[\#L_{\le h} \ge M \right]+\mathbb{P}\left[ L_{>h} \neq \emptyset \right].
\end{eqnarray}
We first upper-bound $\mathbb{P}\left[ \#L_{\le h} \ge M \right]$. Since $L_{\le h} \subset B(z,h)$,
\begin{eqnarray*}
\mathbb{P}\left[ \#L_{\le h} \ge M \right] \le \mathbb{P}\left[ \#\left( \mathcal{N} \cap B(z,h)\right) \ge M \right].
\end{eqnarray*}
By (\ref{E:volball}), $\Vol(B(z,h) \le C e^{dh}$, for some $C>0$ independent of $r_0$. So the random variable $\#\left( \mathcal{N} \cap B(z,h)\right)$ is stochastically dominated by a Poisson law with parameter $C\lambda e^{dh}$, thus, by the Chernoff bound for the Poisson distribution \cite{boundPoisson},
\begin{eqnarray}
\label{E:majorlhmoins}
\mathbb{P}\left[ \#L_{\le h} \ge M \right] \le \frac{e^{-C\lambda e^{dh}}(C\lambda e^{dh+1})^M}{M^M}.
\end{eqnarray}

The second step is to upper-bound $\mathbb{P}\left[ \#L_{>h} \neq \emptyset \right]$. 
\tvc{Recall that $\mathbb{P}_{z'}$ denotes the Palm distribution of $\mathcal{N}$ conditionally on having a point at $z'\in \mathbb{H}^{d+1}$, and that any $z'\in L$ is necessarily outside $B(r_0)$. By Campbell formula \cite{daley-verejones}:
\begin{align}
\mathbb{P}\left[ L_{>h} \neq \emptyset \right] \le \mathbb{E}\left[ \#L_{>h}\right] & 
= \mathbb{E}\left[ 
\sum_{z' \in \mathcal{N}} \ind_{ [z' ,A(z')]^* \cap  B_{S(r_0)}(z,e^{r_0}) \not=\emptyset}
\ind_{d(z',z)>h}
\right] \nonumber
\\
&=\lambda\int_{B(r_0)^c} \mathbb{P}_{z'}\left[ z' \in L_{>h} \right]~dz'.\label{etape3}
\end{align}
We will control the integrand in \eqref{etape3} on each of the following annuli centred at $z$:
\begin{equation}\label{eq:Cn}
\mathcal{C}_n:= \big(B(z,n+1)\setminus B(z,n)\big)\cap B(r_0)^c.
\end{equation}
Consider $z' \in L_{>h}$, by the triangular inequality, denoting by $z^*$ the meeting point of $[z',A(z')]^*$ and $S(r_0)$, $d(z',A(z')) \ge d(z',z^*) \ge d(z',z)-d(z^*,z)$. The hyperbolic law of cosines (\ref{E:cosinelaw}) gives,
\begin{eqnarray*}
&d(z^*,z)&=\arccosh\big(\cosh(r_0)^2-\cos(\widehat{z^*0z})\sinh(r_0)^2\big) \\
&&\le \arccosh\big(\cosh(r_0)^2-\cos(e^{-r_0})\sinh(r_0)^2\big)\\
&&=\arccosh\left(1+\left( 1-\cos(e^{-r_0})\right)\sinh(r_0)^2\right)\\
&&\le C_{\dist}
\end{eqnarray*}
for some $C_{\dist}>0$ independent of $r_0$. Hence, we have on the one hand that, for $z'\in \mathcal{C}_n\cap L_{> h}$,
\begin{align}
d(z',A(z')) \geq  & d(z',z^*)
\geq  d(z,z')-d(z,z^*)
\geq   n-C_{\dist},\label{etape4}
\end{align}where the first inequality uses that the distances are increasing functions along edges thanks to \eqref{E:defofstar}.
On the other hand, $z'\in \mathcal{C}_n\cap L_{> h}$, implies that
\begin{equation}\label{etape5}
h<d(z',z)<n+1.\end{equation}
Let us choose $h$ sufficiently large: $h\geq 2 C_{\dist}+1$. With this choice, every $n >h-1$ also satisfies $n-C_\dist \geq n/2$. Then, for any $z'\in \mathcal{C}_n \cap L_{>h}$, we have that $d(z',A(z'))\geq n/2$, which implies that $z'\notin B(n/2)$. Otherwise, we would have $A(z')=0$ and $d(z',A(z'))=d(z',0)<n/2$, which would be contradictory. Using this in \eqref{etape3},
\begin{align*}
\mathbb{P}\left[ L_{>h} \neq \emptyset \right] \le & \lambda \sum_{n=0}^{+\infty} \int_{\mathcal{C}_n} \mathbb{P}_{z'}\big(z'\in L_{>h} \big)\ dz'\\
\leq & \lambda \sum_{n>h-1} \int_{\mathcal{C}_n \cap B(n/2)^c} \mathbb{P}_{z'}\big(B^+(z',\frac{n}{2})\cap \mathcal{N}=\emptyset \big)\ dz'\\
\leq & \lambda \sum_{n>h-1} \int_{\mathcal{C}_n \cap B(n/2)^c} e^{-\lambda \Vol(B^+(z',n/2))}\ dz'.
\end{align*}Using the estimate \eqref{E:boundbplus} for the volume of balls with $r'\geq n/2$,
\begin{align}
\mathbb{P}\left[ L_{>h} \neq \emptyset \right] \le & \lambda \sum_{n=h}^{+\infty} e^{-\lambda c e^{dn/4}} \Vol\big(\mathcal{C}_n\big)\\
\leq & \lambda \sum_{n=h}^{+\infty} e^{-\lambda c e^{dn/4}+d(n+1)}\leq C e^{-ch}. \label{E:majorlhplus}
\end{align}
Finally, combining (\ref{E:sumprobalh}), (\ref{E:majorlhmoins}) and (\ref{E:majorlhplus}) with $h=-\ln(M/(2C\lambda))/d$ leads to Lemma \ref{Lemma:numberofpoints}.}
\end{proof}

\section{Proof of Propositions \ref{Prop:globalfluct} and \ref{Prop:straightness}}
\label{S:proofglobalfluct}

We first prove Proposition \ref{Prop:globalfluct}.
\paragraph{Step 1:}
Let us fix $p>3d/2$. For any $r_0>0$, $n \in \mathbb{N}$, let us define
\begin{eqnarray*}
S_n(A,r_0):=\sum_{z \in B_{S(r_0)}(u,Ae^{-r_0}) \cap \RST} \MBD_{r_0}^{r_0+n\delta}(z)^{2p}.
\end{eqnarray*}

The strategy of the proof is to construct a family of non-negative random variables \\
$(Y_n^M(A,r_0))_{r_0,A,M \ge 0,n \in \mathbb{N}}$ and $(Y^M(A,r_0))_{r_0,A,M \ge 0}$ such that
\begin{enumerate}[label=(\arabic*)]
\item almost surely, $Y_n^M(A,r_0) \uparrow Y^M(A,r_0)$ when $n \to \infty$ for any $M,A,r_0 \ge 0$;
\item $\sup_{A,r_0} \mathbb{P}\left[Y^M(A,r_0) \ge M \right]=O\left(M^{-2/3}\right)$ when $M \to \infty$;
\item the following implication holds almost surely:
\begin{eqnarray*}
S_n(A,r_0) \le (M \land Y_n^M(A,r_0))A^de^{-2r_0p} \implies S_{n+1}(A,r_0) \le Y_{n+1}^M(A,r_0)A^de^{-2r_0p}.
\end{eqnarray*}
\end{enumerate}

Let us suppose for the moment that such random variables $Y_n^M(A,r_0)$ and $Y^M(A,r_0)$ exist. Let $A,r_0 \ge 0$ and $M \ge 0$. On the event $\{Y^M(A,r_0) \le M\}$, it can be shown by induction that $S_n(A,r_0) \le MA^de^{-2r_0p}$ for any $n \ge 0$. Indeed, $S_0=0$, and if $S_n(A,r_0) \le M A^de^{-2r_0p}$, \dave{then, by (3),
\begin{align*}
S_{n+1}(A,r_0)\leq &  Y^M_{n+1}(A,r_0) A^d e^{-2r_0p}
\leq  Y^M(A,r_0)A^d e^{-2r_0p}\leq M A^d e^{-2r_0 p},
\end{align*}
since we are on the event $\{Y^M(A,r_0) \le M\}$.
This achieves the induction.
}

Thus, for any $A,r_0,M \ge 0$,
\begin{eqnarray}
\label{E:boundM}
\mathbb{P}[S_n(A,r_0) \ge MA^de^{-2r_0p}] \le \mathbb{P}[Y^M(A,r_0) \ge M] \le C M^{-2/3} \text{ by } (2).
\end{eqnarray}
for $M$ large enough and some constant $C>0$ independent of $A,r_0,M$.
It follows that\dave{
\begin{eqnarray}
\label{E:definecprime}
&C':&=\sup_{A,r_0}\mathbb{E}\left[ S_n(A,r_0)A^{-d}e^{2r_0p} \right]=\sup_{A,r_0} \int_0^\infty \mathbb{P}[S_n(A,r_0) \ge MA^{d}e^{-2r_0p}]~dM \nonumber\\
&&\overset{(\ref{E:boundM})}{\le} \int_0^\infty CM^{-2/3}~dM<\infty.
\end{eqnarray}}
Let $K:=\#\left(\mathcal{N} \cap B_{S(r_0)}(u,Ae^{-r_0})\right)$. Let us apply Cauchy-Schwartz with the inner product defined by $\langle X,Y\rangle=\mathbb{E}\left[\sum_i X_iY_i \right]$,
\begin{eqnarray}
\label{E:cauchyschwartzfinal0}
&&\mathbb{E}\left[ \sum_{z \in B_{S(r_0)} (u,Ae^{-r_0}) \cap \RST} \left( \MBD_{r_0}^{r_0+n\delta}(z) \right)^p \right] \nonumber\\
&&\le \mathbb{E}\left[ \sum_{z \in B_{S(r_0)} (u,Ae^{-r_0}) \cap \RST} \left( \MBD_{r_0}^{r_0+n\delta}(z) \right)^{2p} \right]^{1/2}\mathbb{E}\left[ \#B_{S(r_0)}(u,Ae^{-r_0}) \right]^{1/2} \nonumber\\
&&=\mathbb{E}\left[S_n(A,r_0)\right]^{1/2}\mathbb{E}\left[ \#B_{S(r_0)}(u,Ae^{-r_0}) \right]^{1/2} \nonumber\\ &&\overset{(\ref{E:definecprime})}{\le} C'\dave{A^{d}e^{-2r_0p}}\mathbb{E}\left[ \#B_{S(r_0)}(u,Ae^{-r_0}) \cap \RST \right]^{1/2}.
\end{eqnarray}
Let us show that $\mathbb{E}\left[ \#B_{S(r_0)}(u,Ae^{-r_0}) \cap \RST \right] \le CA^d$ for some $C>0$ independent of $A,r_0$. We use the covering of $S(r_0)$ by balls of radius $e^{-r_0}$ introduced by Lemma \ref{Lemma:covering} in Section \ref{S:proofcontrolannulus}. For any $1 \le i \le N(r_0)$, by Proposition \ref{Lemma:numberofpoints} applied with $p=1$, $\mathbb{E}[\#\RST \cap B_{S(r_0)}(z_i,e^{-r_0})] \le C$ for $C$ independent of $r_0$, $z_i$. By Lemma \ref{Lemma:covering}, the number of balls intersecting $B_{S(r_0)}(u,Ae^{-r_0})$ is bounded by $C_{\ball}A^d$. It follows that $\mathbb{E}\left[ \#B_{S(r_0)}(u,Ae^{-r_0}) \cap \RST \right] \le CA^d$.

Thus, resuming to (\ref{E:cauchyschwartzfinal0}),
\begin{eqnarray}
\label{E:cauchyschwartzfinal}
\mathbb{E}\left[ \sum_{z \in B_{S(r_0)} (u,Ae^{-r_0}) \cap \RST} \left( \MBD_{r_0}^{r_0+n\delta}(z) \right)^p \right] \le Ce^{-r_0p}.
\end{eqnarray}

Since $r \mapsto \MBD_{r_0}^r(z)$ is non-decreasing for any $z \in S(r)$,
\begin{eqnarray}
\label{E:peqnstep1}
\sum_{z \in B_{S(r_0)} (u,Ae^{-r_0}) \cap \RST} \left( \MBD_{r_0}^\infty(z) \right)^p =\lim_{n \to \infty} \uparrow \sum_{z \in B_{S(r_0)} (u,Ae^{-r_0}) \cap \RST} \left( \MBD_{r_0}^{r_0+n\delta}(z) \right)^p.
\end{eqnarray}
Proposition \ref{Prop:globalfluct} follows by (\ref{E:cauchyschwartzfinal}) and by monotone convergence theorem.

\paragraph{Step 2:} we build the random variables $Y_n^M(A,r_0)$ and $Y^M(A,r_0)$, as presented in the beginning of Step 1.

Let $A,r_0 > 0$, let $n \in \mathbb{N}$. The strategy is to upper-bound $S_{n+1}$ \tvc{as a function of} $S_n$. Fix $z \in B_{S(r_0)}(u,Ae^{-r_0})$. The quantity $\MBD_{r_0}^{r_0+n\delta}$ takes into account finite backward paths that stop before level $r_0+n\delta$ and those (potentially infinite) that continue after level $r_0+n\delta$. Let us define the random set $\Stop(z)$ as the set of ending points (in the backward direction) of finite paths from $z$ stopping before level $r_0+n\delta$:
\begin{eqnarray*}
\Stop(z):=\{ z'=(r';u') \in \mathcal{N} \cap \mathcal{D}(z),~r_0 \le r' \le r_0+n\delta,~A^{-1}(z')=\emptyset\} \subset \mathcal{N}.
\end{eqnarray*}This definition is similar to \eqref{eq:def_Stop}, but with $r_0+n\delta$ instead of $r_0+\delta$).
By definition of $\MBD_{r_0}^{r_0+n\delta}(z)$ (resp. $\MBD_{r_0}^{r_0+(n+1)\delta}(z)$) (Definition \ref{def:defmbd}),
\begin{eqnarray}
\label{E:redefmbd1}
\MBD_{r_0}^{r_0+n\delta}(z)&=&\max_{z'=(r';u') \in \Stop(z)} \CFD_{r_0}^{r'}(z') \vee \max_{z' \in \mathcal{D}_{r_0}^{r_0+n\delta}(z)} \CFD_{r_0}^{r_0+n\delta}(z')
\end{eqnarray}
and
\begin{eqnarray}
\label{E:redefmbd2}
&&\MBD_{r_0}^{r_0+(n+1)\delta}(z) \\
&&=\max_{z'=(r';u') \in \Stop(z)} \CFD_{r_0}^{r'}(z') \vee \max_{z' \in \mathcal{D}_{r_0}^{r_0+n\delta}(z)} \left(\CFD_{r_0}^{r_0+n\delta}(z')+\MBD_{r_0+n\delta}^{r_0+(n+1)\delta}(z')\right). \nonumber
\end{eqnarray}
For any $p \ge 1$, $a,b \ge 0$ and $t \in [0,1]$, Jensen inequality gives,
\begin{eqnarray}
\label{E:jensen}
(a+b)^p=\left(t\frac{a}{t}+(1-t)\frac{b}{1-t}\right)^p \le t\left(\frac{a}{t}\right)^p+(1-t)\left(\frac{b}{1-t}\right)^p=t^{1-p}a^p+(1-t)^{1-p}\, b^p.
\end{eqnarray}

Applying (\ref{E:jensen}) with $t=1/n^2$ leads to:
\begin{eqnarray}
&&\MBD_{r_0}^{r_0+(n+1)\delta}(z)^{2p} \nonumber\\
&&\overset{(\ref{E:redefmbd2})}{=}\max_{z'=(r';u') \in \Stop(z)} \CFD_{r_0}^{r'}(z')^{2p} \vee \max_{z' \in \mathcal{D}_{r_0}^{r_0+n\delta}(z)} \left(\CFD_{r_0}^{r_0+n\delta}(z')+\MBD_{r_0+n\delta}^{r_0+(n+1)\delta}(z')\right)^{2p} \nonumber\\
&& \le \max_{z'=(r';u') \in \Stop(z)} \CFD_{r_0}^{r'}(z')^{2p} \vee \nonumber\\
&& \max_{z' \in \mathcal{D}_{r_0}^{r_0+n\delta}(z)}\left[\left(1-\frac{1}{n^2} \right)^{1-{2p}}\left(\CFD_{r_0}^{r_0+n\delta}(z')^{2p}\right)+n^{4p-2}\left( \MBD_{r_0+n\delta}^{r_0+(n+1)\delta}(z')^{2p} \right) \right], \nonumber
\end{eqnarray}\dave{by using \eqref{E:jensen}. Now, since $1-2p<0$ and $1-1/n^2\in (0,1)$,}
\begin{eqnarray}
&&\MBD_{r_0}^{r_0+(n+1)\delta}(z)^{2p} \nonumber\\
&&\le \left(1-\frac{1}{n^2} \right)^{1-{2p}}\left[\max_{z'=(r';u') \in \Stop(z)} \left(\CFD_{r_0}^{r'}(z')^{2p}\right) \vee \max_{z' \in \mathcal{D}_{r_0}^{r_0+n\delta}(z)} \left(\CFD_{r_0}^{r_0+n\delta}(z')^{2p}\right) \right] \nonumber\\
&&+n^{4p-2}\max_{z' \in \mathcal{D}_{r_0}^{r_0+n\delta}(z)}\left[ \MBD_{r_0+n\delta}^{r_0+(n+1)\delta}(z')^{2p}\right] \nonumber\\
&&\overset{(\ref{E:redefmbd1})}{=}\left(1-\frac{1}{n^2} \right)^{1-{2p}}\MBD_{r_0}^{r_0+n\delta}(z)^{2p}+n^{4p-2}\max_{z' \in \mathcal{D}_{r_0}^{r_0+n\delta}(z)}\left[ \MBD_{r_0+n\delta}^{r_0+(n+1)\delta}(z')^{2p}\right].
\label{E:globalflucteqn1}
\end{eqnarray}
Summing (\ref{E:globalflucteqn1}) over all $z \in B_{S(r_0)}(u,Ae^{-r_0})$ leads to:
\begin{multline}
\label{E:globalflucteqn2}
 S_{n+1} (A,r_0) \le \left(1-\frac{1}{n^2} \right)^{1-2p}S_n (A,r_0) \\+n^{4p-2} \left( \sum_{\substack{z \in B_{S(r_0)}(u,Ae^{-r_0})\\z' \in \mathcal{D}_{r_0}^{r_0+n\delta}(z)}} \MBD_{r_0+n\delta}^{r_0+(n+1)\delta}(z')^{2p} \right).
\end{multline}

Let us \tvc{condition} on the event $\{S_n(A,r_0) \le MA^de^{-2r_0p}\}$. Then, for any $z \in B_{S(r_0)}(u,Ae^{-r_0})$,
\begin{eqnarray*}
\MBD_{r_0}^{r_0+n\delta}(z)^{2p} \le S_n (A,r_0) \le MA^de^{-2r_0p},
\end{eqnarray*}
so, for any $z' \in \mathcal{D}_{r_0}^{r_0+n\delta}(z)$, $\widehat{z'0z} \le M^{1/{2p}}A^{d/(2p)}e^{-r_0}$. \dave{Denoting by $z_\infty \in \partial \mathbb{H}^{d+1}$ the point of direction $u$, we have}  \begin{eqnarray*}
&\widehat{z'0z_\infty} &\le \widehat{z'0z}+\widehat{z0z_\infty} \\
&&\le M^\frac{d}{2p}A^\frac{d}{2p}e^{-r_0}+Ae^{-r_0}, \text{ since } z \in B_{S(r_0)}(u,Ae^{-r_0})\\
&&\le Ae^{-r_0}\left(M^\frac{d}{2p}+1\right), \text{ since } \frac{d}{2p}\le 1.
\end{eqnarray*}
Therefore, for any $z \in B_{S(r_0)}(u,Ae^{-r_0})$,
\begin{eqnarray}
\label{E:globalflucteqn3}
\mathcal{D}_{r_0}^{r_0+n\delta}(z) \subset B_{S(r_0+n\delta)}\left(u,Ae^{-r_0}\left(M^\frac{d}{2p}+1\right)\right).
\end{eqnarray}
Let us define
\begin{eqnarray}
\label{E:defzn}
Z_n(A,r_0):=\sum_{z' \in B_{S(r_0+n\delta)}(u,Ae^{-r_0}(M^{d/(2p)}+1)) \cap \RST} \MBD_{r_0+n\delta}^{r_0+(n+1)\delta}(z')^{2p}.
\end{eqnarray}
By (\ref{E:globalflucteqn3}),
\begin{eqnarray}
\label{E:globalflucteqn4}
\sum_{\substack{z \in B_{S(r_0)}(u,Ae^{-r_0})\\z' \in \mathcal{D}_{r_0}^{r_0+n\delta}(z)}} \MBD_{r_0+n\delta}^{r_0+(n+1)\delta}(z')^{2p} \le Z_n(A,r_0),
\end{eqnarray}
thus, combining (\ref{E:globalflucteqn2}) and (\ref{E:globalflucteqn4}), on the event \{$S_n \le MA^de^{-2r_0p}$\},
\begin{eqnarray}
\label{E:globalflucteqn5}
S_{n+1}(A,r_0) \le \left(1-\frac{1}{n^2} \right)^{1-2p}S_n(A,r_0)+n^{4p-2}Z_n(A,r_0).
\end{eqnarray}
This upper-bound of $S_{n+1}(A,r_0)$ suggests the following definition of the random variables $Y_n^M(A,r_0)$. We set $Y_0^M(A,r_0):=0$, and for any $n \ge 0$,
\begin{eqnarray*}
Y_{n+1}^M(A,r_0):=\left(1-\frac{1}{n^2}\right)^{1-2p}Y_n^M(A,r_0)+n^{4p-2}A^{-d}e^{2r_0p}Z_n(A,r_0).
\end{eqnarray*}
Let us also define $Y^M(A,r_0):=\lim_{n \to \infty} \uparrow Y_n^M(A,r_0)$ (this is well-defined since $n \to Y_n^M(A,r_0)$ is non-decreasing).\\

We first show that the random variables $Y_n^M(A,r_0)$ verify (3) of Step 1:
for $n \in \mathbb{N}$, on the event $\{S_n(A,r_0) \le (M \land Y_n^M(A,r_0))A^de^{-2r_0p}\}$, by (\ref{E:globalflucteqn5}),
\begin{eqnarray*}
&S_{n+1}(A,r_0) &\le \left(1-\frac{1}{n^2} \right)^{1-2p}S_n(A,r_0)+n^{4p-2}Z_n(A,r_0) \nonumber\\
&&\le \left(1-\frac{1}{n^2} \right)^{1-2p}Y_n^M(A,r_0)A^de^{-2r_0p}+n^{4p-2}Z_n(A,r_0) \nonumber\\
&&=A^de^{-2r_0p}\left[ \left(1-\frac{1}{n^2} \right)^{1-2p}Y_n^M(A,r_0)+n^{4p-2}A^{-d}e^{2r_0p}Z_n(A,r_0)\right] \nonumber\\
&&= A^de^{-2r_0p}Y_{n+1}^M(A,r_0).
\end{eqnarray*}
Thus the random variables $Y_n^M(A,r_0)$ verify (3).\\

We move on to show that $(Y_n^M(A,r_0))_{n,M,A,r_0}$ and $(Y^M(A,r_0))_{M,A,r_0}$ also verify (2) of Step 1. To proceed, we upper-bound $\mathbb{E}[Y_n^M(A,r_0)]$ by induction on $n$.

For any $M,A,r_0,n$, Proposition \ref{Prop:controlannulus} applied for $\theta=Ae^{-r_0}\left(M^{\dave{d}/(2p)}+1\right)$ gives,
\begin{eqnarray}
\label{E:globalflucteqn6}
&\mathbb{E}[Z_n(A,r_0)] &\le C\left(Ae^{-r_0}\left(M^\frac{\dave{d}}{2p}+1\right)\right)^d e^{(d-2p)(r_0+n\delta)}\nonumber\\
&&=\dave{C A^d\left(M^\frac{d}{2p}+1\right)^d e^{-2pr_0+n(d-2p)\delta}.}
\end{eqnarray}
Let us define, for any $n \in \mathbb{N}$,
\begin{eqnarray}
\label{E:straighteqn6}
p(n):=\left(1-\frac{1}{n^2}\right)^{1-2p}, \quad q(n):=n^{4p-2}e^{n(d-2p)\delta}
\end{eqnarray}
and
\begin{eqnarray}
P(n):=\Pi_{k=0}^{n-1} p(k), \quad Q(n):=\sum_{k=0}^{n-1} q(k).
\end{eqnarray}
with the convention $P(0)=1$ and $Q(0)=0$. It can be noticed that
\begin{eqnarray}\label{limit:PnQn}
\lim_{n \to \infty} P(n)<\infty, \quad \lim_{n \to \infty} Q(n)<\infty \text{ since } d-2p<0.
\end{eqnarray}
Let us show by induction on $n$ that $\mathbb{E}[Y_n^M(A,r_0)] \le C(M^{d/(2p)}+1)^dP(n)Q(n)$ for any $n \in \mathbb{N}$. The assertion is clear for $n=0$ and, for $n \ge 0$,
\begin{eqnarray}
&\mathbb{E}[Y_{n+1}^M(A,r_0)]&=
\left(1-\frac{1}{n^2}\right)^{1-2p}\mathbb{E}\left[Y_n^M(A,r_0)\right]+n^{4p-2}A^{-d}e^{2r_0p}\mathbb{E}\left[Z_n(A,r_0)\right] \nonumber\\
&&=p(n)\mathbb{E}[Y_n^M(A,r_0)]+A^{-d}e^{2r_0p-n(d-2p)\delta}q(n)\mathbb{E}\left[Z_n(A,r_0)\right] \nonumber\\
&&\overset{(\ref{E:globalflucteqn6})}{\le}p(n)\mathbb{E}[Y_n^M(A,r_0)]+C\left( M^\frac{d}{2p}+1 \right)^dq(n) \nonumber\\
&&\le C\left( M^\frac{d}{2p}+1 \right)^d[p(n)P(n)Q(n)+q(n)] \text{ by induction hypothesis}\nonumber\\
&&=C\left( M^\frac{d}{2p}+1 \right)^d[P(n+1)Q(n)+q(n)] \nonumber\\
&&\le C\left( M^\frac{d}{2p}+1 \right)^dP(n+1)Q(n+1) \text{ since } P(n+1) \le 1,
\end{eqnarray}
which achieves the induction. Thus, by \eqref{limit:PnQn}, there exists some constant $C>0$ such that, for any $M,A,r_0 \ge 0$, for any $n \in \mathbb{N}$, $\mathbb{E}\left[ Y_n^M(A,r_0) \right] \le C\left( M^{d/(2p)}+1 \right)^d$. By monotone convergence,
\begin{eqnarray*}
\mathbb{E}\left[Y^M(A,r_0)\right] \le C\left( M^{1/(2p)}+1 \right)^d.
\end{eqnarray*}
Thus, for any $M,A,r_0 \ge 0$, Markov inequality gives,
\begin{eqnarray*}
\mathbb{P}\left[ Y^M(A,r_0) \ge M \right] \le \frac{C\left(M^{1/(2p)}+1\right)^d}{M}=O(M^{-2/3})
\end{eqnarray*}
since $2p>3d$. Thus the family of random variables $Y^M(A,r_0)$ verifies (2). This achieves the proof.

\paragraph{Proof of Proposition \ref{Prop:straightness}.}

This is a direct consequence of Proposition \ref{Prop:globalfluct}. Let $\varepsilon>0$ and let us choose $p$ such that $d/p<\varepsilon$. Applying Proposition \ref{Prop:globalfluct} with $A=\pi e^{r_0}$ gives that, for any $r_0 \ge 0$,
\begin{eqnarray}
\label{E:straighteqn1}
\mathbb{E}\left[\sum_{z \in S(r_0) \cap \RST} \left( \MBD_{r_0}^\infty(z) \right)^p \right] \le Ce^{r_0(d-p)}.
\end{eqnarray}
Thus:
\begin{eqnarray*}
&&\mathbb{E}\left[\sum_{n \in \mathbb{N}} \left( \max_{z \in \mathcal{L}_{n}} \MBD_{n}^\infty(z) e^{(1-\varepsilon)n}\right)^p \right] \le \sum_{n \in \mathbb{N}} Ce^{n((d-p)+(1-\varepsilon)p)}<\infty
\end{eqnarray*}
since $\varepsilon>d/p$. Therefore, a.s.,
\begin{eqnarray*}
\lim_{n \to \infty} e^{(1-\varepsilon)n} \max_{z \in \mathcal{L}_{n}} \MBD_{n}^\infty(z) \to 0 \text{ as } n \to \infty.
\end{eqnarray*}
Moreover, $r_0 \mapsto \max_{z \in \mathcal{L}_{r_0}}\MBD_{r_0}^\infty$ is non-increasing, so for any $n \le r_0<n+1$,
\begin{eqnarray}
e^{(1-\varepsilon)r_0}\max_{z \in \mathcal{L}_{r_0}}\MBD_{r_0}^\infty(z) \le \max_{z \in \mathcal{L}_n}\MBD_{n}^\infty(z) e^{(1-\varepsilon)r_0} \le e^{1-\varepsilon}\max_{z \in \mathcal{L}_n}\MBD_{n}^\infty(z) e^{(1-\varepsilon)n},
\end{eqnarray}
thus
\begin{eqnarray*}
\lim_{r_0 \to \infty} e^{(1-\varepsilon)r_0} \max_{z \in \mathcal{L}_{r_0}} \MBD_{r_0}^\infty(z)  \to 0 \text{ as } r_0 \to \infty.
\end{eqnarray*}
Define $R_0$ such that, for any $r_0 \ge R_0$, $e^{(1-\varepsilon)r_0}\max_{z \in \mathcal{L}_{r_0}} \MBD_{r_0}^\infty(z) \le 1/2$. For any $r_0 \ge R_0$, $z \in \mathcal{L}_{r_0}$, $z_1,z_2 \in \mathcal{D}(z)$, defining $r_1:=d(0,z_1)$ and $r_2:=d(0,z_2)$,
\begin{eqnarray}
&\widehat{z_1 0 z_2} &\le \widehat{z_1 0 z}+\widehat{z 0 z_2} \le \CFD_{r_0}^{r_1}(z_1)+\CFD_{r_0}^{r_2}(z_2) \le \MBD_{r_0}^{r_1}(z)+\MBD_{r_0}^{r_2}(z)  \nonumber\\
&&\le 2\MBD_{r_0}^\infty(z) \le 2\max_{z' \in \mathcal{L}_{r_0}}\MBD_{r_0}^\infty(z') \le e^{-(1-\varepsilon)r_0}.
\end{eqnarray}
This achieves the proof of  Proposition \ref{Prop:straightness}.

\appendix

\section{Proof of Proposition \ref{Prop:firstprop}}
\label{S:firstprop}

We first show that the RST is a tree. If the $\RST$ contains some loop $z_0,\cdots,z_n$, then the furthest vertex to the origin in the loop, say $z_i$, must have  two parents, which contradicts the definition of the RST. Moreover, for some given vertex $z \in \mathcal{N}$, the sequence $\left(d\left(A^{(k)}(z),0\right)\right)_k$ is decreasing. In addition, since $\mathcal{N} \cap B(r)$ is finite for any $r \ge 0$, there is no infinite decreasing sequence $\left(d\left(A^{(k)}(z),0\right)\right)_k$. Thus $A^{(k)}(z)=0$ for some finite $k \ge 0$. Therefore, the RST is a connected graph, so it is a tree. \\

We move on to show that the RST is locally finite. Let us assume for the moment that, for any $z=(r;u) \in \mathbb{H}^{d+1}$ and $\rho>0$,
\begin{eqnarray}
\label{E:boundbplus}
\Vol(B^+(z,\rho)) \ge ce^{d(\rho \wedge r)/2}.
\end{eqnarray}
for some $c$ independent of $z,\rho$. \\
For $z_0=(r_0;u_0) \in \mathcal{N} \cup \{0\}$, $z=(r;u) \in \mathbb{H}^d$, let us define
\begin{eqnarray*}
a(z,z_0)=\mathbf{1}_{r>r_0}\mathbf{1}_{B^+(z,d(z,z_0)) \cap \mathcal{N}=\emptyset}.
\end{eqnarray*}
For any $z \in \mathcal{N}$, $z_0=A(z)$ if and only if $a(z,z_0)=1$.
By Campbell formula \cite{daley-verejones},
\begin{eqnarray*}
\mathbb{E}\left[ \#\{z_0 \in \mathcal{N},~\#A^{(-1)}(z_0)=\infty\} \right]&=&\mathbb{E}\left[ \sum_{z_0 \in \mathcal{N}} \mathbf{1}_{\sum_{z \in \mathcal{N}} a(z,z_0)=\infty} \right]
\\
&=& 
\lambda\int_{\mathbb{H}^{d+1}} \mathbb{P}_{z_0}\left[\sum_{z \in \mathcal{N}}a(z,z_0)=\infty\right]~dz_0,
\end{eqnarray*}where we recall that $\mathbb{P}_{z_0}$ is the Palm measure conditioned on having an atom at $z_0$. Thus it suffices to show that $\mathbb{P}_{z_0}\left[\sum_{z \in \mathcal{N}}a(z,z_0)=\infty\right]=0$ for any $z_0 \in \mathbb{H}^{d+1}$. Let $z_0=(r_0;u_0) \in \mathbb{H}^{d+1}$. Note that, if $d(z,z_0) \ge r_0$, then $0 \in B^+(z,d(z,z_0))$ so $a(z,z_0)=0$. Thus, 
\begin{eqnarray*}
&\mathbb{E}\left[ \sum_{z \in \mathcal{N}} a(z,z_0) \right]&=\lambda\int_{\mathbb{H}^{d+1}} \mathbb{E}\left[ a(z,z_0) \right]~dz \le \lambda\int_{\mathbb{H}^{d+1}} \mathbf{1}_{d(z,z_0)<r_0} \mathbb{P}[B^+(z,d(z,z_0)) \cap \mathcal{N}=\emptyset]~dz\\
&&=\lambda\int_{\mathbb{H}^{d+1}} \mathbf{1}_{d(z,z_0)<r_0} \mathbb{P}\left[ \exp(-\lambda \Vol(B^+(z,d(z,z_0))) \right]~dz\\
&&\overset{(\ref{E:boundbplus})}{\le} \lambda\int_{\mathbb{H}^{d+1}} \mathbb{P}\left[ \exp(-\lambda c e^{-d/2~(d(z,z_0)\wedge d(0,z))})\right]~dz\\
&&\overset{(\ref{E:dvolpolar})}{=}\lambda \nu(\mathbb{S}^d) \int_0^{\infty} \exp(-\lambda c e^{-rd/2})\sinh(r)^d~dr <+\infty,
\end{eqnarray*}
where $\nu(\mathbb{S}^d)$ is the surface area of the Euclidean unit ball $\mathbb{S}^d$. Thus $\mathbb{P}\left[\sum_{z \in \mathcal{N}} a(z,z_0)=\infty\right]=0$.\\

\tvc{It remains to show (\ref{E:boundbplus}). Recall that $z=(r;u)$ and that $\rho \ge 0$. Let us introduce $z' =(r-r\wedge \rho ; u)$ and $z''=(r-\frac{r\wedge \rho}{2} ; u)$. The latter point is the center of the geodesic $[z',z]$ (which is here also the segment $[z',z]$ as these points are aligned with $0$). 
We have that 
\begin{equation}\label{etape1}B(z'',\frac{r\wedge \rho}{2})\subset B^+(z,r\wedge \rho)\subset  B^+(z,\rho).\end{equation}
The second inclusion is obvious as $r\wedge \rho\leq \rho$. For the first inclusion, let us consider $x\in B(z'',\frac{r\wedge \rho}{2})$. Since
\[d'(x,z)\leq d(x,z'')+d(z'',z)\leq \frac{r\wedge \rho}{2}+\frac{r\wedge \rho}{2}=r\wedge \rho,\]
we have $B(z'',\frac{r\wedge \rho}{2})\subset B(z,r\wedge \rho)$. Moreover,
\[d(x,0)\leq d(x,z'')+d(z'',0)\leq \frac{r\wedge \rho}{2}+\big(r-\frac{r\wedge \rho}{2}\big)=r,\]
so $B(z'',\frac{r\wedge \rho}{2})\subset B(r)$. From these two inclusions, we deduce the first inclusion in \eqref{etape1}.\\
As a consequence, 
\begin{align*}
\Vol\big(B^+(z,\rho)\big)\geq   &  \Vol\big(B\big(z'',\frac{r\wedge \rho}{2} \big)\big)
=  \Vol\big(B\big(\frac{r\wedge \rho}{2} \big)\big)= \nu(\mathbb{S}^d)  \int_0^{\frac{r\wedge \rho}{2}} \sinh^d(t) \ dt \geq c(d) e^{d\frac{r\wedge \rho}{2}},
\end{align*}where the last inequality holds when $r$ and $\rho$ are sufficiently far from $0$. 
}


It remains to show that the geodesics $[z,A(z)]$ for $z \in \mathcal{N}$ do not cross a.s. in the bi-dimensional case ($d=1$). Let us suppose that there are no two points $z_1$, $z_2$ with $d(0,z_1)=d(0,z_2)$ (this happens with probability $0$). Let $z_1=(r_1;u_1),z_2=(r_2;u_2) \in \mathcal{N}$ and let us set $A(z_1):=(r'_1;u'_1)$, $A(z_2):=(r'_2;u'_2)$. Suppose that $[z_1,A(z_1)]$ and $[z_2,A(z_2)]$ meet at some point $P_{hyp}:=(r_{hyp};u_{hyp})$. We have $r'_1<r_{hyp}<r_2$, thus by definition of the parent, $d(z_2,A(z_2))<d(z_2,A(z_1))$. Then
\begin{eqnarray}
&d(z_2,P_{hyp})+d(P_{hyp},A(z_2))&=d(z_2,A(z_2))<d(z_2,A(z_1)) \nonumber\\
&&\le d(z_2,P_{hyp})+d(P_{hyp},A(z_1)),
\end{eqnarray}
so $d(P_{hyp},A(z_2))<d(P_{hyp},A(z_1))$. On the other hand, interchanging $z_1$ and $z_2$ in the previous calculation leads to $d(P_{hyp},A(z_1))<d(P_{hyp},A(z_2))$. This is a contradiction. Therefore $[z_1,A(z_1)] \cap [z_2,A(z_2)]=\emptyset$. This achieves the proof of Proposition \ref{Prop:firstprop}.

\bibliographystyle{plain}


\begin{thebibliography}{10}

\bibitem{baccelli}
F.~Baccelli and C.~Bordenave.
\newblock The radial spanning tree of a Poisson point process.
\newblock {\em The Annals of Applied Probability}, 17(1):305--359, 2007.

\bibitem{baccelli2013semi}
F.~Baccelli, D.~Coupier, and V.C. Tran.
\newblock Semi-infinite paths of the two-dimensional radial spanning tree.
\newblock {\em Advances in Applied Probability}, 45(4):895--916, 2013.

\bibitem{benjamini}
I.~Benjamini and O.~Schramm.
\newblock Percolation in the hyperbolic plane.
\newblock {\em Journal of the American Mathematical Society}, 14(2):487--507,
  2001.

\bibitem{percolationbeyondZd}
I.~Benjamini and O.~Schramm.
\newblock Percolation beyond $\mathbb{Z}^d$, many questions and a few answers.
\newblock {\em Electronic Communications in Probability}, 1:71--82, 1996.

\bibitem{socialnetworks}
M.~Bogun{\'a}, F.~Papadopoulos, and D.~Krioukov.
\newblock Sustaining the internet with hyperbolic mapping.
\newblock {\em Nature communications}, 1:62, 2010.

\bibitem{calka2018mean}
P.~Calka, A.~Chapron, and N.~Enriquez.
\newblock Mean asymptotics for a Poisson-Voronoi cell on a Riemannian manifold.
\newblock {\em International Mathematics Research Notices}, 2021(7):5413--5459, 2021.

\bibitem{cannon}
J.W. Cannon, W.J. Floyd, R.~Kenyon, and W.R. Parry.
\newblock Hyperbolic geometry.
\newblock {\em Flavors of geometry}, 31:59--115, 1997.


\bibitem{chavel}
I. Chavel.
\newblock {\em Riemannian Geometry: A Modern Introduction}, Second Edition. Cambridge studies in advanced mathematics 98.
\newblock Cambridge University Press, 2006.



\bibitem{coletti2013radial}
C.~Coletti and L.A. Valencia.
\newblock Scaling limit for a family of coalescing radial random paths absorbed at the origin.
\newblock {\em Journal of Mathematical Physics}, 63:033303, 2022.



\bibitem{coupier2019directed}
D.~Coupier, J.F. Marckert, and V.C. Tran.
\newblock Directed, cylindric and radial brownian webs.
\newblock {\em Electronic Journal of Probability}, 24(20):1--48, 2019.

\bibitem{dsftobw}
D.~Coupier, K.~Saha, A.~Sarkar, and V.C. Tran.
\newblock The 2d-directed spanning forest converges to the Brownian web.
\newblock {\em The Annals of Probability}, 49(1):435--484, 2021.

\bibitem{dsf}
D.~Coupier and V.C. Tran.
\newblock The 2d-directed spanning forest is almost surely a tree.
\newblock {\em Random Structures \& Algorithms}, 42(1):59--72, 2013.

\bibitem{daley-verejones}
D.J. Daley and D. Vere-Jones,
\newblock {\em An introduction to the theory of point processes}, Vol. II, Second edition. Springer, New York, 2008.

\bibitem{HyperbolicDSF}
L.~Flammant.
\newblock The directed spanning forest in the hyperbolic space.
\newblock {\em arXiv preprint arXiv:1909.13731}, 2019.

\bibitem{gangopadhyay}
S.~Gangopadhyay, R.~Roy, and A.~Sarkar.
\newblock Random oriented trees: a model of drainage networks.
\newblock {\em The Annals of Applied Probability}, 14(3):1242--1266, 2004.

\bibitem{howard}
C.D. Howard and C.M. Newman.
\newblock Geodesics and spanning trees for {E}uclidean first-passage
  percolation.
\newblock {\em Ann. Probab.}, 29(2):577--623, 2001.

\bibitem{hutchcroft2019percolation}
T.~Hutchcroft.
\newblock Percolation on hyperbolic graphs.
\newblock {\em Geometric and Functional Analysis}, 29(3):766--810, 2019.

\bibitem{leeriemannian}
J.M. Lee.
\newblock {\em Riemannian manifolds: an introduction to curvature}, volume 176.
\newblock Springer Science \& Business Media, 2006.

\bibitem{lyons2017probability}
R.~Lyons and Y.~Peres.
\newblock {\em Probability on trees and networks}, volume~42.
\newblock Cambridge University Press, 2017.

\bibitem{boundPoisson}
M. Mitzenmacher and E.~Upfal.
\newblock {\em Probability and computing: randomized algorithms and
  probabilistic analysis}.
\newblock Cambridge university press, 2005.

\bibitem{paupert}
J.~Paupert.
\newblock Introduction to hyperbolic geometry.
\newblock {\em Arizona State University Lecture Notes}, 2016.

\bibitem{penrose2003random}
M.~Penrose.
\newblock {\em Random geometric graphs}, volume~5.
\newblock Oxford university press, 2003.

\bibitem{pete2014probability}
G.~Pete.
\newblock Probability and geometry on groups.
\newblock {\em Lecture notes for a graduate course, present version is at
  http://www. math. bme. hu/\~{} gabor/PGG. pdf}, 2014.

\bibitem{ramsayrichtmyer}
A. Ramsay and R.D. Richtmyer.
\newblock Introduction to Hyperbolic Geometry. 
\newblock Springer, New York, 1995.

\bibitem{ratcliffe}
J.G. Ratcliffe.
\newblock Foundations of Hyperbolic Manifolds.
\newblock Springer. Graduate texts in Mathematics 149, second edition, 2006.

\bibitem{roy2016random}
R.~Roy, K.~Saha, and A.~Sarkar.
\newblock Random directed forest and the Brownian web.
\newblock In {\em Annales de l'Institut Henri Poincar{\'e}, Probabilit{\'e}s et
  Statistiques}, volume~52, pages 1106--1143. Institut Henri Poincar{\'e},
  2016.

\bibitem{tykesson}
J.~Tykesson.
\newblock The number of unbounded components in the Poisson Boolean model of continuum percolation in hyperbolic space.
\newblock {\em Electronic Journal of Probability}, 12:1379--1401, 2007.


\end{thebibliography}

\end{document}